\documentclass[reqno]{amsart}
\usepackage[T1]{fontenc}
\usepackage{amsmath,amscd,amsfonts,amsthm,amssymb,latexsym,mathrsfs,mathtools,pgfplots,float,mathdots}
\usepackage{bbm} %allows \mathhbb for lower cases
\pgfplotsset{compat=1.15}
\usepackage{tikz}
\usetikzlibrary{matrix,arrows,patterns}
\usepackage[justification=centering]{caption}
\usepackage{shuffle}
\usepackage{mathtools}
\usepackage{mathrsfs}
\usepackage{comment}
\usepackage{enumitem}
\usepackage{hyperref}
\hypersetup{
    unicode=false,          % non-Latin characters in AcrobatÔøΩs bookmarks
    pdftoolbar=true,        % show AcrobatÔøΩs toolbar?
   pdfmenubar=true,        % show AcrobatÔøΩs menu?
   pdffitwindow=false,     % window fit to page when opened
   pdfstartview={FitH},    % fits the width of the page to the window
   pdftitle={},    % title
   pdfauthor={Author},     % author     pdfsubject={Subject},   % subject of the document
    pdfcreator={Creator},   % creator of the document
    pdfproducer={Producer}, % producer of the document
   pdfkeywords={keyword1} {key2} {key3}, % list of keywords
   pdfnewwindow=true,      % links in new window
    colorlinks=true,       % false: boxed links; true: colored links
    linkcolor=blue,          % color of internal links
    citecolor=black,        % color of links to bibliography
    filecolor=magenta,      % color of file links
    urlcolor=black    % color of external links
    }
\usepackage{cleveref}

\numberwithin{equation}{section}
\theoremstyle{plain}
   \newtheorem{theorem}{Theorem}[section]
   \newtheorem{proposition}[theorem]{Proposition}
   \newtheorem{lemma}[theorem]{Lemma}
   \newtheorem{corollary}[theorem]{Corollary}
   
   \newtheorem{conjecture}[theorem]{Conjecture}

\theoremstyle{definition}
   \newtheorem{definition}[theorem]{Definition}
   \newtheorem{example}[theorem]{Example}

\theoremstyle{remark}
   \newtheorem{remark}[theorem]{Remark}

\newcommand{\NN}{\mathbb{N}}

\newcommand{\one}{\hat{1}}
\newcommand{\zero}{\hat{0}}

\newcommand{\EE}{\mathcal{E}}

\newcommand{\QQ}{\mathbb{Q}}
\newcommand{\RR}{\mathbb{R}}

\newcommand{\sym}{\mathfrak{S}}

\def\newop#1{\expandafter\def\csname #1\endcsname{\mathop{\rm
#1}\nolimits}}

\newop{diag}
\newop{AB}
\newop{DB}
\newop{B}
\newop{AT}
\newop{DT}
\newop{peak}
\newop{rank}
\newop{des}
\newop{Sym}
\newop{Int}
\newop{Orb}
\newop{Re}
\newop{inv}
\newop{Hom}

\title[]{On  chain polynomials of geometric lattices}
\author[P.~Br\"and\'en]{Petter Br\"and\'en}
\address{Department of Mathematics, KTH Royal Institute of Technology, SE-100 44 Stockholm,
Sweden}
\author[L.~Saud]{Leonardo Saud Maia Leite}
\keywords{Totally nonnegative matrix, geometric lattice, chain polynomial, real-rooted polynomial, paving matroid, perfect matroid design, single-element extension}

\thanks{PB is a Wallenberg Academy Scholar
  supported by the Knut and Alice Wallenberg
  Foundation, and the G\"oran Gustafsson foundation.}

\begin{document}
\definecolor{ududff}{rgb}{0.30196078431372547,0.30196078431372547,1.}
\definecolor{xdxdff}{rgb}{0.49019607843137253,0.49019607843137253,1.}
\definecolor{red}{rgb}{1,0,0}
\definecolor{ffqqff}{rgb}{1,0,1}
\definecolor{zzttqq}{rgb}{0.6,0.2,0}
\definecolor{ududff}{rgb}{0.30196078431372547,0.30196078431372547,1}
\definecolor{cqcqcq}{rgb}{0.7529411764705882,0.7529411764705882,0.7529411764705882}

\begin{abstract}
Athanasiadis and Kalampogia-Evangelinou recently conjectured that the chain polynomial of any geometric lattice has only real zeros. We verify this conjecture for families of geometric lattices including perfect matroid designs, Dowling lattices, and for a class of geometric lattices that contains  all lattices of flats of paving matroids. 

We also investigate how the conjecture behaves with respect to certain operations such as direct products, ordinal sums and single-element extensions.
\end{abstract}

\maketitle
\thispagestyle{empty}
\setcounter{tocdepth}{2}
\tableofcontents
\newpage

\section{Introduction}
Chain enumeration in partially ordered sets (posets) is a central topic in enumerative and algebraic combinatorics \cite[Chapter 3]{stanley2011enumerative}. Polynomials arising from chain enumeration have been extensively studied. 
The \emph{chain polynomial} of a finite poset $P$ is defined as  
\begin{equation}\label{chainpol}
    c_P \coloneqq \sum_{k \geq 0} c_k (P) t^k,
\end{equation}
where $c_k(P)$ is the number of $k$-element chains (totally ordered subposets) of $P$. Equivalently, $c_P$ is the $f$-polynomial of the order complex of $P$.

A key question in this area is whether the chain polynomials of various classes of posets have only real zeros -- that is, whether they are real-rooted. For finite distributive lattices, this question is equivalent to the Neggers-Stanley conjecture for natural labelings, originally proposed by Neggers \cite{neggers1978representations} in the 1970s and later disproved by Stembridge \cite{stembridge2007counterexamples}. 
Nevertheless, significant progress has been made in establishing real-rootedness for specific families of posets.
The search for classes of posets with real-rooted chain polynomials remains an active area of research.  For instance, Brenti and Welker \cite{brenti2008f} proved that the chain polynomials of the face lattices of simplicial polytopes are real-rooted, and they conjectured that the same holds for all polytopes. This conjecture was later confirmed for cubical polytopes by Athanasiadis \cite{athanasiadis2021face}. More recently, the authors have shown that the chain polynomials of shellable subspace lattices, partition lattices,  and $r$-cubical lattices are real-rooted \cite{branden2024totally}.

Athanasiadis and Kalampogia-Evangelinou \cite{athanasiadis2022chain} proposed the following conjecture:
\begin{conjecture}\cite[Conjecture 1.2]{athanasiadis2022chain}
\label{conj1}
    The chain polynomial $c_L$ is real-rooted for every geometric lattice $L$.
\end{conjecture}
In support of this conjecture, Athanasiadis, Kalampogia-Evangelinou, and Douvropoulos \cite{athanasiadis2023two,athanasiadis2022chain} verified it for several cases, including subspace lattices, partition lattices of types $A$ and $B$, lattices of flats of near-pencils, and Boolean lattices. The main goal of this paper is to extend their results by proving the conjecture for additional classes of geometric lattices.

\begin{itemize}[leftmargin=4mm]
\item Section \ref{triang} shows that triangular semimodular lattices, which include the well-studied class of perfect matroid designs \cite{Deza}, have real-rooted chain polynomials.
\item Section \ref{dowling} establishes the real-rootedness of chain polynomials for Dowling lattices, which generalize the partition lattices of types $A$ and $B$. This result extends earlier work by Athanasiadis and Kalampogia-Evangelinou \cite{athanasiadis2022chain}.
\item Section \ref{gen-TN-posets} introduces a new construction on posets inspired by the structure of paving lattices (the lattices of flats of paving matroids). We prove that posets constructed via this method using $\mathrm{TN}$-posets, introduced by the authors in \cite{branden2024totally}, have real-rooted chain polynomials. Notably, this includes paving lattices themselves. 
\item Section \ref{SEE} explores single-element extensions, demonstrating that geometric lattices arising from principal single-element extensions of Boolean lattices, as well as from single truncations of Boolean algebras, have real-rooted chain polynomials. These extensions are important in Matroid Theory, since any lattice of flats of a matroid can be constructed via a finite sequence of single-element extensions of a Boolean algebra.
\item Section \ref{sec-counterexample} shows that for each $n \geq 3$, there  is a geometric lattice of rank $n$ whose $h$-polynomial is not interlaced by the Eulerian polynomial $A_n$. This disproves \cite[Conjecture 5.1]{athanasiadis2022chain}.
\end{itemize}

\section{Preliminaries}
\subsection{Interlacing sequences of polynomials}
A polynomial $f \in \mathbb{R}[t]$ is called \emph{real-rooted} if all of its zeros are real. Suppose $f$ and $g$ are real-rooted polynomials with positive leading coefficients, and that the zeros of $f$ and $g$ are
\[
    \cdots \leq \alpha_3 \leq \alpha_2 \leq \alpha_1  \ \  \mbox{ and } \ \ \cdots \leq \beta_3 \leq \beta_2 \leq \beta_1,
\]
respectively. We say that the zeros of $g$ \emph{interlace} those of $f$ if 
\[
 \cdots \leq \beta_3 \leq \alpha_3 \leq \beta_2 \leq \alpha_2 \leq \beta_1 \leq \alpha_1, 
\]
and write $g \preceq f$. In particular the degrees of $f$ and $g$ differ by at most one. By convention we also write $0 \preceq f$ and $f \preceq 0$ for any real-rooted polynomial $f$. A sequence $\{f_i\}_{i=0}^n$ of real-rooted polynomials with nonnegative coefficients is said to be \emph{interlacing} if $f_i \preceq f_j$ for all $i<j$, and we let $\mathcal{P}_n$ be the set of all interlacing sequences $\{f_i\}_{i=0}^n$ of polynomials. 

Recall that a matrix with real nonnegative entries is $\mathrm{TP}_2$ if all of its $2\times 2$-minors are nonnegative. If $A=(a_{ij}(t))$ is an $(m+1) \times (n+1)$ matrix with entries in $\RR[t]$, then it defines a map on polynomial sequences as usual
$$
\{f_j\}_{i=0}^n \longmapsto \{g_i\}_{i=0}^m, \quad \mbox{ where } \quad g_i= \sum_{i=0}^n a_{ij} f_j. 
$$ 

\begin{lemma}
\label{basint}
\hfill
    \begin{enumerate}[label=(\roman*)]
        \item\label{basint:1}\hypertarget{basint:1}{Let} $A = (A_{ij})$ be an $(m+1) \times (n+1)$ matrix with entries in $\RR$. Then $A \colon \mathcal{P}_n \longrightarrow \mathcal{P}_m$ 
        if and only if $A$ is $\mathrm{TP}_2$.
        \item\label{basint:2}\hypertarget{basint:2}{If}  $h \preceq f$ and $h \preceq g$, then $h \preceq f + g$.
        \item\label{basint:4}\hypertarget{basint:2}{If}  $f \preceq h$ and $g \preceq h$, then $f + g \preceq h$.
        \item\label{basint:3}\hypertarget{basint:2}{Suppose} $f, g \in \mathbb{R}_{\geq 0}[t]$, $\deg (f) = \deg (g) + 1$, and $g \preceq f$. If $\lambda \geq 0$ is such that the leading coefficient of $f - \lambda t  g$ is positive, then $g \preceq f - \lambda t g \preceq f$.
     \end{enumerate}
\end{lemma}
\begin{proof}
    The first result is \cite[Corollary 7.8.6]{branden2015unimodality}. The second and third results are well known, and can be found in e.g. \cite[Lemma 2.6]{borcea2010multivariate}. 
    
 To prove $(iv)$ we may assume that $f$ and $g$ have no zeros in common, since we may factor out any common linear factor. We prove $g \preceq f - \lambda t g$, the proof of $f - \lambda t g \preceq f$ follows similarly.  Notice that all the zeros of $g$ are simple, since otherwise $f$ and $g$ would have a common zero by interlacing. Let $\alpha_n < \cdots < \alpha_2< \alpha_1$ be the zeros of $g$, and let $F_\lambda= f - \lambda t g$. Then $(-1)^iF_\lambda(\alpha_i) = (-1)^if(\alpha_i) >0$ for each $1 \leq i \leq n$, by interlacing. By the intermediate value theorem, $F_\lambda$ has a zero in each interval $(\alpha_{i+1}, \alpha_i)$, $1 \leq i <n$. Since $F_\lambda(\alpha_1)<0$ and the leading coefficient of $F_\lambda$ is positive, there must be a zero of $F_\lambda$ in the interval $(\alpha_1, \infty)$. By the same reasoning $F_\lambda$ must have a zero in the interval $(-\infty, \alpha_n)$. Hence $g \preceq f - \lambda t g$.
\end{proof}

\subsection{\texorpdfstring{$\mathrm{TN}$}--posets}
For undefined poset terminology we refer to \cite{stanley2011enumerative}. The notion of quasi-rank uniform posets was introduced in \cite{branden2024totally}. Suppose $P$ is a locally finite poset with a least element $\zero$. Define a \emph{quasi-rank function} $\rho : P \to \NN$ by 
$$
    \rho(x) = \max\{ k : \zero =x_0 <x_1 < \cdots <x_k=x\}. 
$$
The \emph{quasi-rank} of $P$ is $\rho(P)= \sup\{ \rho(x) : x \in P\} \in \NN\cup\{\infty\}$ and the \emph{quasi-rank generating polynomial} of a finite poset $P$ is defined as
\[
    f_P = \sum_{x \in P} t^{\rho(x)}.
\]
For $x \leq y$ in $P$, let $[x,y] = \{ z \in P \colon x \leq z \leq y \}$. We say that $P$ is \emph{quasi-rank uniform} if for any $x, y\in P$ with $\rho(x)=\rho(y)$ and each $0 \leq k \leq \rho(x)$,
$$
    |\{ z \in [\zero, x] : \rho(z) = k \}| = |\{ z \in [\zero, y] : \rho(z) = k \}|.
$$
If $P$ is a graded poset, then we omit the ``quasi'' in all definitions above.

\begin{figure}[H]
\centering
\begin{tikzpicture}[line cap=round,line join=round,>=triangle 45,x=1cm,y=1cm,scale=0.9]
    \clip(-5.988811187214801,-1.9199521854769663) rectangle (6.195417344425535,4);
    \draw [line width=1pt] (-1,0)-- (0,-1);
    \draw [line width=1pt] (0,-1)-- (1,0);
    \draw [line width=1pt] (3,0)-- (0,-1);
    \draw [line width=1pt] (5.5,0)-- (0,-1);
    \draw [line width=1pt] (-3,0)-- (0,-1);
    \draw [line width=1pt] (-2,1.5)-- (-3,0);
    \draw [line width=1pt] (-1,1.5)-- (-1,0);
    \draw [line width=1pt] (-1,3)-- (-1,1.5);
    \draw [line width=1pt] (-1,3)-- (-2,1.5);
    \draw [line width=1pt] (-5.5,0)-- (0,-1);
    \draw [line width=1pt] (-5.5,0)-- (-1,3);
    \draw [line width=1pt] (1,0)-- (1,1.5);
    \draw [line width=1pt] (1,1.5)-- (1,3);
    \draw [line width=1pt] (1,3)-- (2,1.5);
    \draw [line width=1pt] (2,1.5)-- (3,0);
    \draw [line width=1pt] (5.5,0)-- (1,3);
    \draw (-0.020566157853083845,-1.081019629462914) node[anchor=north west] {0};
    \draw (-6.1,0.5) node[anchor=north west] {1};
    \draw (-3.6,0.5) node[anchor=north west] {1};
    \draw (-1.6,0.5) node[anchor=north west] {1};
    \draw (1.1280394893070957,0.5) node[anchor=north west] {1};
    \draw (3.1887731503885943,0.5) node[anchor=north west] {1};
    \draw (5.660527459718589,0.5) node[anchor=north west] {1};
    \draw (-2.7794718789731228,1.5709081147157344) node[anchor=north west] {2};
    \draw (-0.6849556988575014,1.554016855198673) node[anchor=north west] {2};
    \draw (0.30,1.554016855198673) node[anchor=north west] {2};
    \draw (2.2,1.565277694876714) node[anchor=north west] {2};
    \draw (-0.9946287900036282,3.357965935030875) node[anchor=north west] {3};
    \draw (0.970387733814522,3.3635963548698955) node[anchor=north west] {3};
    \begin{scriptsize}
        \draw [fill=black] (-1,0) circle (2pt);
        \draw [fill=black] (-3,0) circle (2pt);
        \draw [fill=black] (0,-1) circle (2pt);
        \draw [fill=black] (1,0) circle (2pt);
        \draw [fill=black] (3,0) circle (2pt);
        \draw [fill=black] (5.5,0) circle (2pt);
        \draw [fill=black] (-2,1.5) circle (2pt);
        \draw [fill=black] (-1,1.5) circle (2pt);
        \draw [fill=black] (-1,3) circle (2pt);
        \draw [fill=black] (1,1.5) circle (2pt);
        \draw [fill=black] (1,3) circle (2pt);
        \draw [fill=black] (2,1.5) circle (2pt);
        \draw [fill=black] (-5.5,0) circle (2pt);
    \end{scriptsize}
\end{tikzpicture}
\caption{A quasi-rank uniform poset $P$, and the corresponding quasi-rank function. Its quasi-rank generating polynomial is $f_P = 1 + 6t + 4t^2 + 2t^3$.}
\label{fig:quasi-rank-uniform}
\end{figure}
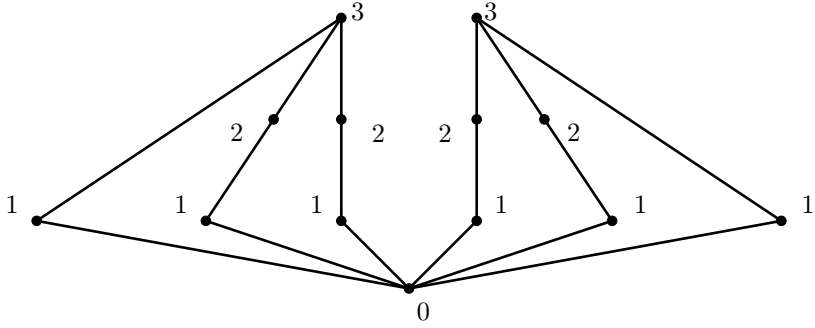

We associate a matrix $R(P)$ to any quasi-rank uniform poset $P$ as follows. If $x \in P$, $\rho(x)=n$, and $0 \leq k \leq n$, let 
$$
    r_{n,k}= r_{n,k}(P)=|\{ z \in [\zero, x] : \rho(z) = k \}|, 
$$
and let $R(P)= (r_{n,k})_{n,k=0}^N$, where $N = \rho(P)$. We say that $P$ is a \emph{$\mathrm{TN}$-poset} if the matrix $R(P)$ is totally nonnegative, that is, if all the minors of $R(P)$ are nonnegative.

Let $N \in \mathbb{N} \cup \{ \infty \}$, and let $\{ R_n\}^N_{n=0}$ be a sequence of monic polynomials such that $R_n$ has degree $n$ for each $n$. Write $R_n = \sum_{k=0}^n r_{n,k} t^k$.  The sequence $\{ R_n\}^N_{n=0}$ (and the matrix $R = (r_{n,k})_{n,k=0}^N$)  is called \emph{resolvable} if there is a matrix $(\lambda_{n,k})_{0\leq k\leq n <N}$ of nonnegative numbers, and an array $(R_{n,k})_{0\leq k \leq n \leq N}$ of monic polynomials such that 
\begin{itemize}
    \item $R_{n,0}=R_n$ and $R_{n,n} =t^n$ for all $0\leq n \leq N$,
    \item $t^k$ divides  $R_{n,k}$ for all $0 \leq k \leq n \leq N$, and 
    \item if $0\leq k \leq n <N$, then 
    \begin{equation}\label{r-pasc}
        R_{n+1,k}=R_{n+1,k+1}+ \lambda_{n,k} R_{n,k}. 
    \end{equation}
\end{itemize}
In particular, \eqref{r-pasc} implies
\begin{equation}
\label{r-pasc2}
    R_{n+1,k} = t^{n+1} + \sum_{j=k}^n \lambda_{n,j} R_{n,j}.
\end{equation}

In \cite[Theorem 2.6]{branden2024totally} it is shown that a sequence of monic polynomials is resolvable if and only if its corresponding matrix $R$ is $\mathrm{TN}$. Hence, if $R_n$ is the quasi-rank generating polynomial of an interval $[\hat{0}, x]$ of a quasi-rank uniform poset $P$, with $\rho(x) = n$, then $P$ is a $\mathrm{TN}$-poset if and only if the sequence $\{ R_n \}^N_{n=0}$ is resolvable.

The \emph{chain polynomials} associated to $R$ are the polynomials $\{p_n\}_{n=0}^N$ defined by $p_0=1$, and 
\begin{equation}
\label{polreceq}
    p_n= t \sum_{k=0}^{n-1}r_{n,k}\cdot p_k, \ \ \ 0<n \leq N. 
\end{equation}  

Let $\RR_N[t]$ be the linear space of all polynomials in $\RR[t]$ of degree at most $N$. The \emph{subdivision operator} associated to $R$ is the linear map $\EE : \RR_N[t] \to \RR_N[t]$ defined by $\EE(t^n)= p_n$ for each $0 \leq n \leq N$. By \cite[Proposition 5.2]{branden2024totally}, if $R=R(P)$ for a quasi-rank uniform poset $P$ with a least element $\zero$, then the polynomials $p_n$ may be expressed as
\begin{equation}
\label{pnchain}
    p_n = \sum_{j \geq 1} |\{\zero =x_0 < x_1<\cdots < x_j=x\}|t^j,
\end{equation}
where $x$ is any element in $P$ for which $\rho(x)=n$.

\section{Triangular semimodular lattices and perfect matroid designs}\label{triang} 
In this section we will prove that triangular semimodular lattices, and thus perfect matroid designs, are $\mathrm{TN}$-posets. A locally finite lattice $L$ is \emph{semimodular} if it is graded, and its rank function $\rho$ satisfies
\[
    \rho(x) + \rho(y) \geq \rho(x \wedge y) + \rho(x \vee y)
\]
for every $x, y \in L$. Moreover, if
\[
    \rho(x) + \rho(y) = \rho(x \wedge y) + \rho(x \vee y)
\]
for every $x, y \in L$, then $L$ is \emph{modular}.

If $L$ is a semimodular lattice that is also atomistic, i.e., each element except $\hat{0}$ is the join of a set of atoms, then $L$ is called a \emph{geometric lattice}.

\begin{definition}
    Let $P$ be a locally finite poset with a least element $\hat{0}$ such that each finite interval in $P$ is graded.  Then $P$ is  \emph{triangular} if there is a function $M : \{(i,j) \in \NN \times \NN : i < j \} \to \NN$ such that any interval $[x,y]$ of $P$ with $\rho(x)=i$ and $\rho(y)=j$ has exactly $M(i,j)$  maximal chains. 
\end{definition}
    
 Triangular posets were studied in \cite[Section 9]{triangular-posets} by Doubilet, Rota and Stanley. In particular, modular and semimodular triangular lattices   were classified in \cite[Theorem 9.1]{triangular-posets} and \cite[Theorem 9.2]{triangular-posets}, respectively.   
 
 Chains and Boolean algebras are clearly triangular.   Let $V_n(q)$ be an $n$-dimensional vector space over a finite field with $q$ elements. The \emph{projective geometry} $\mathbb{B}_n(q)$ is the  lattice of all linear subspaces of $V_n(q)$ partially ordered by inclusion \cite[Section 6.1]{oxley2006matroid}.
     By e.g. \cite[Proposition 6.1.8]{oxley2006matroid},
     $\mathbb{B}_n(q)$ is triangular.

    The \emph{affine geometry} $\mathbb{A}_n(q)$ is the lattice of all affine subspaces of $V_n(q)$ partially ordered by inclusion \cite[Section 6.2]{oxley2006matroid}.
By e.g. \cite[Proposition 6.2.5]{oxley2006matroid}  $\mathbb{A}_n(q)$ is triangular. 
    
 Chains, Boolean algebras, and projective geometries are examples of modular triangular lattices, while affine  geometries are examples of geometric (but not necessarily modular) triangular lattices. 
 
 The \emph{truncation} of a finite poset $P$ of quasi rank $r>0$ is the poset $\tau(P)$ obtained from $P$ by removing all elements of quasi rank $r-1$. Iterated truncations of (affine) projective geometries are examples of triangular geometric lattices, because truncations preserve the properties of being a geometric lattice and being triangular.     
    
By Proposition \ref{prop-triangular} below it follows that  any triangular poset is rank uniform.  
For our purposes it is easier to work with the following alternative definition of triangular posets. 
\begin{proposition}
\label{prop-triangular}
Let $P$ be a locally finite poset with a least element $\hat{0}$ such that each finite interval in $P$ is graded. Then $P$ is  triangular if and only if there exists a function $N : \{ (i,j,k)\in \NN^3 : i <j<k \} \to \NN$ such that any interval $[x,y]$ of $P$ with $\rho(x)=i$ and $\rho(y)=k$ has exactly $N(i,j,k)$  elements of rank $j$.
\end{proposition}

\begin{proof}
Suppose $P$ is triangular, $x<y$,  and $0 \leq  \rho(x)=i < j < k=\rho(y)$. Then  
$$
M(i,k) = \sum_{\substack{x <z < y\\ \rho(z)=j}} M(i,j)M(j,k),
$$
since each maximal chain of $[x,y]$ contains a unique element of rank $j$. Hence 
$$
N(i,j,k):= |\{ z \in [x,y] \colon \rho(z) = j \}| = \frac {M(i,k)} {M(i,j)M(j,k)}
$$
does not depend on the choices of $x$ and $y$. 

Conversely if $N$ exists, then  
$$
M(i,k)= N(i,i+1,k)N(i+1,i+2,k) \cdots N(k-2,k-1,k).
$$
\end{proof}
A geometric lattice $L$ is called a \emph{perfect matroid design} if for all $x, y \in L$, 
$$
\rho(x) = \rho(y) \ \ \mbox{ implies } \ \ |\{ a \leq x : \rho(a)=1\}| = |\{ a \leq y : \rho(a)=1\}|, 
 $$
 see \cite{Deza}. 
\begin{proposition}\label{alleq}
Let $L$ be a geometric lattice. Then the following are equivalent
\begin{itemize}
\item[(i)] $L$ is triangular, 
\item[(ii)] $L$ is rank uniform,
\item[(iii)] $L$ is a perfect matroid design. 
\end{itemize}
\end{proposition}
\begin{proof}
By Proposition \ref{prop-triangular}, $(i)$ implies $(ii)$. Clearly $(ii)$ implies $(iii)$. Also $(iii)$ implies $(i)$ by \cite[Proposition 2.2.3]{Deza}. 
\end{proof}
Hence projective geometries and truncations of Boolean algebras are perfect matroid designs, but there are more \cite{Deza}.

Next we will define an element $R_P$ in the incidence algebra of $P$, that will be used to prove resolvability for triangular semimodular lattices. 

Recall that the \emph{incidence algebra}  of a locally finite poset $P$, over the field $\QQ$, is the set $I(P)$ of all function $f : \mathrm{Int}(P) \to \QQ$, where $\mathrm{Int}(P)$ is the set of all finite and closed intervals of $P$.  The multiplication in $I(P)$ is  defined by 
$$
(fg)(x,y)= \sum_{x \leq z \leq y} f(x,z)g(z,y). 
$$
The element $\zeta \in I(P)$ is the function defined by $\zeta(x,y)=1$ for all $x\leq y$. The \emph{M\"obius function}, $\mu$, is the inverse of $\zeta$ in $I(P)$. 

 Suppose $P$ is a locally finite poset with a least element $\zero$, and let $f_P \in I(P)$ be defined by 
$$
f_P(x,y) = \sum_{z \leq y}t^{\rho(z)},
$$
where $\rho$ is the quasi-rank function of $P$. 
Notice that $f_P(x,y)$ does not depend on $x$. 
Define a function $R_P$ in $I(P)$ by 
$
R_P= f_P \mu. 
$
It follows that $R_P(x,y)$ is a monic polynomial in $t$ of degree $\rho(y)$. 
Also $R_P(y,y)= f_P(y,y)$, and 
\begin{align*}
    R_P(\zero, y) &= \sum_{x \leq y} f_P(\zero,x)\mu(x,y)= \sum_{x \leq y} \left(\sum_{z \leq x}t^{\rho(z)}\right)\mu(x,y)  \\
    &= \sum_{z \leq y}t^{\rho(z)} \sum_{z \leq x \leq y}\mu(x,y)= t^{\rho(y)}.
\end{align*}

If $P$ is triangular, let $T(P)$ be the subset of $I(P)$ given by 
$$
T(P) = \{f \in I(P) : f(x,y)=f(x',y') \mbox{ if } \rho(x)=\rho(x') \mbox{ and } \rho(y)=\rho(y')\}. 
$$

\begin{lemma}\label{triR}
If $P$ is triangular, then $R_P \in T(P)$. 
\end{lemma}

\begin{proof}
We prove that $R_P(x,y)$ only depends on the ranks of $x$ and $y$ by  induction on $m=\rho(y)-\rho(x)$. The case when $m=0$ follows from the fact that $P$ is rank uniform. 

Suppose true for all integers smaller than  $m$. For $x\leq y$ with $\rho(y)-\rho(x)=m$, 
$$
f_P(x,y)= (R_P\zeta)(x,y)= \sum_{x\leq z\leq y} R_P(x,z)= R_P(x,y) + \sum_{x \leq z<y}R_P(x,z),  
$$
and hence 
$$
R_P(x,y)=  \sum_{z \leq y}t^{\rho(z)}- \sum_{x \leq z<y}R_P(x,z). 
$$
Clearly $\sum_{z \leq y}t^{\rho(z)}$ only depends on $\rho(y)$ by Proposition \ref{prop-triangular}. By induction the terms $R_P(x,z)$ only depends on the ranks of $x$ and $z$, since $\rho(z)-\rho(x)<m$. The lemma now follows by Proposition \ref{prop-triangular} and induction. 
\end{proof}

\begin{proposition}\label{altformR}
Let $L$ be a locally finite lattice.  For all $x\leq y$, 
$$
R_L(x,y) =  \sum_{z \vee x = y} t^{\rho(z)}.
$$
Moreover 
\begin{itemize}
\item[(i)]
if $L$ is semimodular, then $t^{\rho(y)-\rho(x)}$ divides $R_L(x,y)$, and 
\item[(ii)] if $L$ is geometric, then $t^{\rho(y)-\rho(x)+1}$ does not divide $R_L(x,y)$. 
\end{itemize}
\end{proposition}

\begin{proof}
Let $S_L \in I(L)$ be defined by $S_L(x,y)= \sum_{z \vee x = y} t^{\rho(z)}$. Given $x\leq y$, for each $z \leq y$ there is a unique $w \in [x,y]$ such that $z\vee x=w$. It follows that  
$$
 f_L(x,y)=  \sum_{z \leq y} t^{\rho(z)} = \sum_{x\leq w \leq y} \left(\sum_{z \vee x =w}t^{\rho(z)}\right)= (S_L \zeta)(x,y). 
 $$
 Hence $f_L= S_L\zeta$, and thus $S_L=f_L\zeta^{-1}=  f_L\mu=R_L$. 
 
If $L$ is semimodular and $z \vee x = y$, then 
 $$
 \rho(x)+\rho(z) \geq \rho(x \vee z)+\rho(x \wedge z)= \rho(y)+\rho(x \wedge z)\geq \rho(y), 
 $$
 and hence $\rho(z) \geq \rho(y)-\rho(x)$. It follows that all terms of $R_L(x,y)$ are divisible by $t^{\rho(y)-\rho(x)}$. 
 
For $(ii)$ it suffices to prove that if $L$ is geometric, then  there exists $z \in L$ for which $z \vee x=y$ and $\rho(y)=\rho(x)+\rho(z)$. To see this let $A$ be  a minimal set (with respect to inclusion) of atoms whose join is equal to $x$. Then $|A|=\rho(x)$. The set $A$ may be extended to a minimal set of atoms $A \cup B$ whose join is $y$. Then $|A \cup B| =\rho(y)$. It follows that the join of $B$ has the desired properties of $z$. 
\end{proof}

\begin{theorem}\label{RP-rec}
Let $L$ be a semimodular lattice. If  $x <x'\leq y$, where $x'$ covers $x$, then 
$$
R_L(x,y)=R_L(x',y)- \sum_{w}R_L(x,w),
$$
where the sum is over all $w \in L$ for which $y$ covers $w$, $x \leq w$ and $x' \not \leq w$.   
\end{theorem}

\begin{proof}
   If  $x' \vee z=y$, then either $x \vee z = y$ or $x \vee z < y$. Hence, by Proposition \ref{altformR}, 
  $$
  R_L(x,y) =  \sum_{\substack{z \leq y\\ x' \vee z = y}} t^{\rho(z)} -  \sum_{\substack{x' \vee z = y \\ x \vee z < y}} t^{\rho(z)} = R_L(x',y)- \sum_{\substack{x' \vee z = y \\ x \vee z < y}} t^{\rho(z)}. 
  $$  
   If $x' \vee z = y$  and $w=x \vee z<y$, then by semimodularity,
        \begin{align*}
            \rho(x') + \rho(w) &\geq \rho(\overbrace{x' \vee (x \vee z)}^{= y}) + \rho(\overbrace{x' \wedge (x \vee z)}^{\geq x}) \smash{\text{\quad\raisebox{-0.5\baselineskip}{,}}}\\
            &\geq \rho(y) + \rho(x)
        \end{align*}
        which implies
       $ \rho(y) - \rho(w) \leq \rho(x') - \rho(x) = 1$.
    It follows that $x' \vee z = y$  and $w=x \vee z<y$ if and only if     
         $y$ covers $w$, $x \leq w$  and  $x' \nleq w$. Hence  
         $$
         \sum_{\substack{x' \vee z = y \\ x \vee z < y}} t^{\rho(z)}= \sum_{w}R_L(x,w),
$$
where the sum is over all $w \in L$ for which $y$ covers $w$, $x \leq w$ and $x' \not \leq w$.  
\end{proof}

\begin{theorem}\label{rankunigeo}
If $L$ is a  triangular semimodular lattice, then $L$ is a $\mathrm{TN}$-poset. 
\end{theorem}

\begin{proof}
We prove that the matrix $R(L)$ is resolvable. By Lemma \ref{triR}  the polynomials 
$$
R_{n,k}= R_L(x,y), 
$$
where $x \leq y$, $\rho(x)=n-k$ and $\rho(y)=n$, are well defined. Hence $R_{n,0}=R_n$, $R_{n,n}=t^n$, and $t^k$ divides $R_{n,k}$ for all $k \leq n$ by Proposition \ref{altformR}.  Moreover by Theorem~\ref{RP-rec}, 
$$
R_{n+1,k+1} = R_{n+1,k}- \lambda_{n,k} R_{n,k},
$$
for some numbers $\lambda_{n,k}$, 
which proves that $R(L)$ is resolvable. 
\end{proof}
By Theorem \ref{rankunigeo} and \cite[Theorem 6.6]{branden2024totally} we deduce the following result. 
\begin{corollary}
Let $L$ be a finite triangular semimodular lattice. Then the chain polynomial of $L$ is real-rooted, and all of its zeros lie in the interval $[-1,0]$. 
\end{corollary}

In \cite[Theorem 9.1]{triangular-posets} it was proved that any semimodular triangular lattice $L$ may be decomposed into geometric semimodular lattices as follows. If $P=(X,\leq_P)$ and $Q=(Y,\leq_Q)$ are posets on disjoint sets $X$ and $Y$, define the \emph{ordinal sum} $P\oplus Q =(X\cup Y, \leq)$ to be the poset with relation $z \leq w$ if and only if either $z \leq_P w$, $z \leq_Q w$, or $z \in P$ and $w \in Q$. If $L$ is a triangular semimodular lattice, then Theorem 9.1 in \cite{triangular-posets} says that there are triangular geometric lattices $L_1, L_2, \ldots, L_m$ such that 
\begin{equation}\label{LLL}
L = \left(L_1 \setminus \{\one_1\} \right) \oplus \left( L_2 \setminus \{\one_2\} \right) \oplus \cdots \oplus \left( L_{m-1} \setminus \{\one_{m-1}\} \right) \oplus L_m,  
\end{equation}
where $\one_i$ is the largest element of $L_i$. This suggests that $\mathrm{TN}$-posets with $\zero$ and $\one$ are closed under $\oplus$. This follows from the next theorem. 
\begin{theorem}\label{RSTT}
Suppose $\{R_n\}_{n=0}^N$  and  $\{S_n\}_{n=0}^M$ are resolvable sequences of polynomials for which $S_n(0)=1$ for each $0\leq n \leq M$. Then the sequence of polynomials $\{T_n\}_{n=0}^{N+M+1}$ defined by 
\begin{equation}\label{RST}
T_n = \begin{cases}
R_n, &\mbox{ if } n \leq N \\
t^{N+1}S_{n-N-1} + R_N, &\mbox{ if } n > N.
\end{cases}
\end{equation}
is resolvable.
\end{theorem}

\begin{proof}
Suppose $\{R_{n,k}\}_{n,k}$ and $\{\lambda_{n,k}\}_{n,k}$ resolve $\{R_n\}_{n=0}^N$, and that $\{S_{n,k}\}_{n,k}$ and $\{\mu_{n,k}\}_{n,k}$ resolve $\{S_n\}_{n=0}^M$. 

Define $\{T_{n,k}\}_{0\leq k \leq n \leq N+M+1}$ and $\{\gamma_{n,k}\}_{n,k}$ as follows. If $n \leq N$, then $T_{n,k}= R_{n,k}$ and $\gamma_{n,k} = \lambda_{n,k}$. If $n=N+1$, then 
$T_{N+1,0}= t^{N+1}+ R_N$ and $T_{N+1,k}= t^{N+1}$ for $k>0$. Also  $\gamma_{N,0}=1$ and $\gamma_{N,k}=0$ for $k>0$. Moreover for $m > 0$, $T_{N+1+m,0}= t^{N+1}S_m+R_N$,
$$
T_{N+1+m,k} = t^{N+1}S_{m,k}, \ \ k>0,   
$$
and $\gamma_{N+m,k}= \mu_{m,k}$. 

To prove that $\{T_{n,k}\}_{0\leq k \leq n \leq N+M+1}$ and 
$\{\gamma_{n,k}\}_{n,k}$ resolve $\{T_n\}_{n=0}^{N+M+1}$ it remains to prove \eqref{r-pasc} for $\{T_{n,k}\}_{0\leq k \leq n \leq N+M+1}$. For $n<N$, \eqref{r-pasc} follows from that of $\{R_{n,k}\}_{n,k}$. For $n=N$ and $k=0$,  \eqref{r-pasc} reads 
$$
t^{N+1}+R_N= t^{N+1} + R_N,  
$$
and for $n=N$ and $k>0$, \eqref{r-pasc} reads 
$$
t^{N+1}= t^{N+1}+ 0\cdot R_{N,k}. 
$$
For $n=N+m$, $m >0$, and $k=0$, \eqref{r-pasc} reads
$$
t^{N+1}S_m +R_N= t^{N+1}S_{m,1} + \mu_{m,0}(t^{N+1}S_{m-1}+R_N). 
$$
Since $S_n(0)=1$ for all $n$, it follows that $\mu_{m,0}=1$. Hence the above equation becomes 
$$
t^{N+1}S_m +R_N= t^{N+1}(S_m-S_{m-1})+ t^{N+1} S_{m-1} +R_N. 
$$
For $n=N+m$, $m >0$, and $k>0$, \eqref{r-pasc} reads
$$
t^{N+1}S_{m,k} = t^{N+1}S_{m,k+1}+\mu_{m,k}t^{N+1}S_{m-1,k},
$$
which follows from the resolvability of $\{S_{n,k}\}_{n,k}$. 
\end{proof}
\begin{corollary}\label{oplus}
Suppose $P$ and $Q$ are finite $\mathrm{TN}$-posets such that $P$ has a largest element $\one$ and $Q$ has a least element $\zero$. Then $P\oplus Q$, $(P\setminus \{\one\}) \oplus Q$ and $(P\setminus \{\one\}) \oplus (Q\setminus \{\zero\})$ are $\mathrm{TN}$-posets. 
\end{corollary}

\begin{proof}
It suffices to prove that $P\oplus Q$ is $\mathrm{TN}$, since $\mathrm{TN}$-posets are closed under quasi-rank selection \cite[Proposition 5.9]{branden2024totally}. 

If the rank generating polynomials of $P$ and $Q$ are $\{R_n\}_{n=0}^N$  and  $\{S_n\}_{n=0}^M$, respectively, then the rank generating polynomial of $P \oplus Q$ is given by the polynomials $\{T_n\}_{n=0}^{N+M+1}$ defined in \eqref{RST}. Also $S_n(0)=1$ for all $n$, since $Q$ has a least element. The proof now follows from Theorem \ref{RSTT}. 
\end{proof}

\section{Dowling lattices}
\label{dowling}
In this section, we prove that Dowling lattices have real-rooted chain polynomials. These lattices were studied by Dowling in \cite{dowling1973q} and \cite{dowlingGroups}, where the interval structures, Möbius functions, and characteristic polynomials were determined, as well as Stirling-like identities and recursions for their Whitney numbers of the first and second kinds.

Let $G$ be a finite group. The \emph{rank-}$n$ \emph{Dowling geometry associated to $G$}, $Q_n(G)$, was first described by Dowling in \cite{dowlingGroups} and can be understood as a ``$G$-labeled'' version of $\Pi_n^A$, the usual partition lattice on $[n]$. It is defined as follows: a $G$\emph{-labeled set} is a pair $(B, \alpha)$ where $B$ is a set and $\alpha \colon B \to G$ is a function. The $G$-labeled sets $(B, \alpha)$ and $(B, \beta)$ are said to be \emph{equivalent} if there is an element $g \in G$ such that $\beta (b) = g \alpha(b)$ for all $b \in B$. This defines an equivalence relation on the set of all functions $\alpha \colon B \to G$, and we denote the equivalent class of $(B, \alpha)$ by $[B, \alpha]$. A \emph{partial} $G$\emph{-partition} is a set $\gamma = \{ [S_1, \alpha_1], \ldots, [S_k, \alpha_k] \}$ such that $S_1, \ldots, S_k$ are pairwise disjoint nonempty subsets of $[n]$. A partial ordering $\leq$ on the set of all $G$-partitions of $[n]$ was defined by Dowling in \cite{dowlingGroups} and this poset is a geometric lattice denoted as $Q_n(G)$.

For example, $Q_n(\{ 0 \}) = \Pi_n^A$, where $\{ 0 \}$ denotes the trivial group, and  $Q_n(G) = \Pi_n^B$ if $G$ is a group of order $2$, where $\Pi_n^B$ is the type $B$ partition lattice \cite{vic1997}.

Recall that, given a poset $P$, the \emph{dual poset} $P'$ is the poset on the same ground set as $P$ such that $y \leq_{P'} x$ if and only if $x \leq_P y$. Clearly $c_P= c_{P'}$. 

Sometimes it is more convenient working with $P'$ instead of $P$. For example, the poset $Q_n(G)'$ is rank uniform \cite[Theorem 2]{dowlingGroups}, and its rank generating polynomial may be written as
$$
    R^{G'}_n  = \sum_{i=0}^n W_m(n,i)t^i,
$$
where the order of $G$ is $m$, $W_m(n,0) = W_m(n,n) = 1$ for every $m,n \in \mathbb{Z}_{\geq 0}$, and the coefficients $W_m(n,i)$ satisfy the recursion\footnote{There is a typo in the recursion given in \cite{dowlingGroups}, $i-1$ should be $i$.} \cite[Theorem 7]{dowlingGroups}
\begin{equation}
\label{dowling-rec2}
    W_m(n,i) = W_m(n-1,i-1) + (1+mi) \cdot W_m(n-1,i).
\end{equation}

Since $Q_n(G)'$ is rank uniform, to show that the chain polynomial of $Q_n(G)$ is real-rooted, it suffices to prove that, for each nonnegative number $N$, the sequence of polynomials $\{ R^{G'}_n \}_{n=0}^N$ is resolvable.

In fact, define the linear operator $\alpha \colon \mathbb{R}[t] \to \mathbb{R}[t]$ on the basis $\{ 1, t, t^2, \ldots \}$ by
\[
    \alpha (t^i) = (1 + mi)t^i
\]
for each $i \in \NN$. We claim that
\begin{equation}
\label{dowling-resolvable}
    R_n^{G'} = (t + \alpha)^n 1
\end{equation}
for each nonnegative integer $n$. The proof follows by induction, the case $n=0$ being trivial. Assume that $R_{n-1}^{G'} = (t + \alpha)^{n-1} 1$ for some $n-1 > 0$. Then
\begin{align*}
    (t + \alpha)^n 1 &= (t + \alpha) R_{n-1}^{G'}\\
                     &= \sum_{i=1}^n W_m(n-1,i-1) t^i + \sum_{i=0}^{n-1} (1 + mi) \cdot W_m(n-1,i) t^i = R_n^{G'} \smash{\text{\quad\raisebox{1.0\baselineskip}{,}}}\\
\end{align*}
where the third equality follows from \eqref{dowling-rec2}, which verifies \eqref{dowling-resolvable} by induction. Therefore, $\{ R^{G'}_n \}_{n=0}^N$ is resolvable by \cite[Theorem 2.6]{branden2024totally}, from which the next theorem follows. 

\begin{theorem}\label{dowling-TN}
The dual of any Dowling lattice is a $\mathrm{TN}$-poset. 
\end{theorem}
The following corollary now follows from  \cite[Theorem 6.6]{branden2024totally}.

\begin{corollary}
\label{dowling-th}
    The chain polynomial $c_{Q_n(G)}$ of $Q_n(G)$ is $[-1,0]$-rooted, and the zeros of $c_{Q_n(G)}$ interlace the zeros of $c_{Q_{n+1}(G)}$ for every $n \geq 0$. Moreover, if $S$ is a set of nonnegative integers, then the chain polynomial of the rank selected subposet $Q_n(G)_S$ is real-rooted and all of its zeros lie in the interval $[-1,0]$ for any natural number $n$.
\end{corollary}

Dowling lattices are examples of geometric lattices whose duals are rank uniform. Such lattices were studied in \cite{Bry}, where they were called \emph{upper combinatorially uniform}. From Propositions \ref{prop-triangular} and \ref{alleq} it follows that perfect matroid designs are upper combinatorially uniform. We offer the following conjecture.
\begin{conjecture}
The dual of  any upper combinatorially uniform geometric lattice is a $\mathrm{TN}$-poset. 
\end{conjecture}

\section{Generalized paving lattices}
\label{gen-TN-posets}

In this section we introduce a construction on posets and prove that posets that arise from this construction using $\mathrm{TN}$-posets have real-rooted chain polynomials. We use this construction to prove that chain polynomials of paving lattices (lattices of flats of paving matroids) are real-rooted.

\begin{definition}
\label{def-construction}
    Let $P$ be a locally finite poset with a least element $\zero$ and with quasi-rank function $\rho$. Fix $1 \leq d < n$ and $y \in P$ such that $\rho(y) = n$. Suppose $\mathcal{H} \subseteq \langle y \rangle = \{ z \in P \colon z \leq y \}$ satisfies
    \begin{enumerate}[label=(\roman*)]
        \item $y \notin \mathcal{H}$,
        \item $\rho(h) \geq d$ for all $h \in \mathcal{H}$,
        \item $\mathcal{H}$ is an anti-chain, and 
        \item for all $x < y$ in $P$ with $\rho(x) \leq d-1$ there exists $h \in \mathcal{H}$ such that $x \leq h$. 
    \end{enumerate}
    We define the poset $P(\mathcal{H}, y, d)$ as the poset of quasi-rank $d+1$ obtained by adjoining $\mathcal{H}$ as the set of elements of quasi-rank $d$ and $y$ as the only element of rank $d+1$ to the set $\{ x \in \langle y \rangle \colon \rho(x) \leq d-1 \}$.
\end{definition}

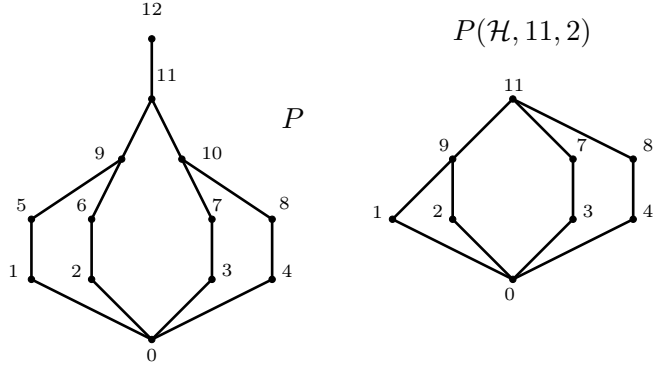
\begin{figure}[H]
\centering
\begin{tikzpicture}[line cap=round,line join=round,>=triangle 45,x=1cm,y=1cm,scale=0.4]
    \clip(-21.322578453639682,-2.8501106263440157) rectangle (1.7,10.5);
    %\draw[help lines,step=1.0] (-21,-3) grid (1,9);
    \draw [line width=1pt] (-20,0)-- (-16,-2);
    \draw [line width=1pt] (-18,0)-- (-16,-2);
    \draw [line width=1pt] (-18,0)-- (-18,2);
    \draw [line width=1pt] (-14,0)-- (-14,2);
    %\draw [line width=1pt] (-16,-2)-- (-16,0);
    \draw [line width=1pt] (-14,0)-- (-16,-2);
    \draw [line width=1pt] (-12,0)-- (-16,-2);
    %\draw [line width=1pt] (-16,2)-- (-16,0);
    \draw [line width=1pt] (-20,2)-- (-20,0);
    \draw [line width=1pt] (-12,2)-- (-12,0);
    \draw [line width=1pt] (-20,2)-- (-17,4);
    \draw [line width=1pt] (-17,4)-- (-18,2);
    %\draw [line width=1pt] (-17,4)-- (-16,2);
    %\draw [line width=1pt] (-16,2)-- (-15,4);
    \draw [line width=1pt] (-15,4)-- (-14,2);
    \draw [line width=1pt] (-15,4)-- (-12,2);
    %\draw [line width=1pt] (-18,6)-- (-17,4);
    %\draw [line width=1pt] (-18,6)-- (-15,4);
    %\draw [line width=1pt] (-14,6)-- (-15,4);
    %\draw [line width=1pt] (-14,6)-- (-17,4);
    \draw [line width=1pt] (-16,6)-- (-17,4);
    \draw [line width=1pt] (-16,6)-- (-15,4);
    %\draw [line width=1pt] (-18,8)-- (-18,6);
    %\draw [line width=1pt] (-18,8)-- (-16,6);
    %\draw [line width=1pt] (-16,6)-- (-14,8);
    %\draw [line width=1pt] (-14,8)-- (-14,6);
    \draw [line width=1pt] (-16,8)-- (-16,6);
    %\draw [line width=1pt] (-16,8)-- (-18,6);
    %\draw [line width=1pt] (-18,8)-- (-17,9);
    %\draw [line width=1pt] (-17,9)-- (-16,8);
    %\draw [line width=1pt] (-16,8)-- (-15,9);
    %\draw [line width=1pt] (-15,9)-- (-14,8);
    \draw [line width=1pt] (-8,2)-- (-4,0);
    \draw [line width=1pt] (-6,2)-- (-4,0);
    %\draw [line width=1pt] (-4,2)-- (-4,0);
    \draw [line width=1pt] (-2,2)-- (-4,0);
    \draw [line width=1pt] (0,2)-- (-4,0);
    \draw [line width=1pt] (0,2)-- (0,4);
    \draw [line width=1pt] (0,4)-- (0,4);
    \draw [line width=1pt] (-8,2)-- (-6,4);
    \draw [line width=1pt] (-6,2)-- (-6,4);
    %\draw [line width=1pt] (-5,4)-- (-4,6);
    \draw [line width=1pt] (-4,6)-- (-2,4);
    \draw [line width=1pt] (-2,2)-- (-2,4);
    %\draw [line width=1pt] (-4,2)-- (-2,4);
    \draw [line width=1pt] (-6,4)-- (-4,6);
    \draw [line width=1pt] (-4,6)-- (0,4);
    \draw (-12,6) node[anchor=north west] {\large $P$};
    \draw (-6.3,9) node[anchor=north west] {\large $P(\mathcal{H},11,2)$};
    \begin{scriptsize}
        \draw [fill=black] (-20,0) circle (3pt);
        \draw (-20.6,0.28733458478751817) node {1};
        \draw [fill=black] (-18,0) circle (3pt);
        \draw (-18.5,0.28733458478751817) node {2};
        %\draw [fill=xdxdff] (-16,0) circle (5pt);
        %\draw (-15.4,0.31764806508830595) node {3};
        \draw [fill=black] (-14,0) circle (3pt);
        \draw (-13.5,0.2974390782211141) node {3};
        \draw [fill=black] (-12,0) circle (3pt);
        \draw (-11.5,0.2671255979203263) node {4};
        \draw [fill=black] (-16,-2) circle (3pt);
        \draw (-15.997510414134629,-2.5) node {0};
        \draw [fill=black] (-20,2) circle (3pt);
        \draw (-18.332338097147282,2.4699051664442373) node {6};
        \draw [fill=black] (-18,2) circle (3pt);
        \draw (-20.332338097147282,2.4699051664442373) node {5};
        \draw [fill=black] (-14,2) circle (3pt);
        \draw (-13.832338097147282,2.4699051664442373) node {7};
        %\draw [fill=ududff] (-16,2) circle (5pt);
        %\draw (-16,3) node {7};
        \draw [fill=black] (-12,2) circle (3pt);
        \draw (-11.591951277086803,2.5507411139130047) node {8};
        \draw [fill=black] (-17,4) circle (3pt);
        \draw (-17.72537879127953,4.207878037022736) node {9};
        \draw [fill=black] (-15,4) circle (3pt);
        \draw (-14,4.207878037022736) node {10};
        \draw [fill=black] (-16,6) circle (3pt);
        \draw (-15.5,6.8) node {11};
        %\draw [fill=ududff] (-14,6) circle (5pt);
        %\draw (-13.3,6.55212051361699) node {13};
        %\draw [fill=ududff] (-18,6) circle (5pt);
        %\draw (-19,6.511702539882607) node {11};
        %\draw [fill=ududff] (-18,8) circle (5pt);
        %\draw (-18.7,8.219361930160318) node {14};
        \draw [fill=black] (-16,8) circle (3pt);
        \draw (-16,9) node {12};
        %\draw [fill=ududff] (-14,8) circle (5pt);
        %\draw (-13.3,8.219361930160318) node {16};
        %\draw [fill=ududff] (-17,9) circle (5pt);
        %\draw (-16.917019316591855,9.7) node {17};
        %\draw [fill=ududff] (-15,9) circle (5pt);
        %\draw (-14.916329616739864,9.7) node {18};
        \draw [fill=black] (-8,2) circle (3pt);
        \draw (-8.5,2.2273973240379354) node {1};
        \draw [fill=black] (-6,2) circle (3pt);
        \draw (-6.5,2.257710804338723) node {2};
        %\draw [fill=ududff] (-4,2) circle (5pt);
        %\draw (-4,2.5) node {3};
        \draw [fill=black] (-2,2) circle (3pt);
        \draw (-1.5,2.277919791205915) node {3};
        \draw [fill=black] (0,2) circle (3pt);
        \draw (0.5,2.267815297772319) node {4};
        \draw [fill=black] (-4,0) circle (3pt);
        \draw (-4.1,-0.5) node {0};
        \draw [fill=black] (-6,4) circle (3pt);
        \draw (-6.286747374522942,4.480699359729826) node {9};
        \draw [fill=black] (-2,4) circle (3pt);
        \draw (-1.710101548955979,4.460490372862634) node {7};
        \draw [fill=black] (0,4) circle (3pt);
        \draw (0.5,4.460490372862634) node {8};
        \draw [fill=black] (-4,6) circle (3pt);
        \draw (-3.9630587347218746,6.4510755792810315) node {11};
    \end{scriptsize}
\end{tikzpicture}
\caption{The Hasse diagrams of $P$ and $P(\mathcal{H},11,2)$. Observe that $P$ is a rank uniform poset of rank $5$, and $P(\mathcal{H},11,2)$ is a graded poset (not rank uniform) of rank $3$.}
\label{fig:gen-quasi-rank-uniform}
\end{figure}

\begin{theorem}
\label{main-theorem}
    If $P$ is a $\mathrm{TN}$-poset with a least element $\zero$,  and $(\mathcal{H}, y, d)$ is as in Definition \ref{def-construction}, then the chain polynomial of  $P(\mathcal{H}, y, d)$ is real-rooted and its zeros lie in the interval $[-1,0]$.
\end{theorem}

We start by proving two lemmas that will be helpful for the proof of Theorem~\ref{main-theorem}.

Given a poset $P$ with a least element $\zero$, and with quasi-rank function $\rho$, let $P_I = \{ x \in P \colon \rho(x) \in I \}$ denote the \emph{quasi-rank selected subposet} induced by the set $I$, where $\emptyset\neq I \subseteq \{0, 1, 2, \ldots, \rho(P)\}$. For the case where $\rho(P)$ is finite and $I = \{0, \ldots, \rho(P)-2, \rho(P)\}$, we denote $P_I$ by $\tau(P)$ and call it the \emph{truncation} of $P$.

\begin{lemma}
\label{main-rec}
    Suppose $P$ is a finite poset of quasi-rank $d+1$, that has a least element $\zero$ and a largest element $\hat{1}$.  
Then
    \[
        c_P  = c_{\tau(P)} + t\sum_{\substack{h \in P \\ \rho(h) = d}} c_{\langle h \rangle} .
    \]
\end{lemma}

\begin{proof}
We prove that for any poset $Q$ of quasi-rank $n$, with a least element $\zero$, 
\begin{equation}\label{PQ}
c_Q= c_{Q_I} + t \sum_{\substack{h \in P \\ \rho(h) = n}}c_{[\zero, h)}, 
\end{equation}
where $I=\{0,\ldots, n-1\}$ and $[\zero, h)= \{ x \in Q : x <h\}$. This will prove the lemma since if $Q= P \setminus \{\one\}$, then 
$c_P=(1+t)c_Q$. 

For any chain $C$ of $Q$, either $C$ does not contain any element of quasi-rank $n$ and then $C$ is recorded by $c_{Q_I}$. Otherwise $C$ contains exactly one element $h$ of quasi-rank $n$, and then $C$ is counted by $tc_{[\zero, h)}$. The lemma follows. 
\end{proof}

\begin{lemma}
\label{extra-lemma}
    Let $P$ be a finite $\mathrm{TN}$-poset of quasi-rank $d+1$, and suppose $P$ has a least element $\zero$ and a largest element $\one$. Then $c_{\langle h \rangle}  \preceq c_{\tau(P)} $ for each $h \in P$ such that $\rho(h) = d$.
\end{lemma}

\begin{proof}
    Since
    \[
        c_{\langle h \rangle}  = (1+t) c_{[\hat{0}, h)}  \quad \mbox{and} \quad c_{\tau(P)}  = (1+t) c_{P_J} ,
    \]
    where $J = \{0, 1, \ldots, d-1 \}$, it suffices to prove $c_{[\hat{0}, h)}  \preceq c_{P_J} $ for each $h \in P$ such that $\rho(h) = d$. In fact, let $\{ R_{n,j}  \}_{0 \leq j \leq n \leq d+1}$ be the polynomials that resolve $R(P) = (r_{n,k})_{n,k=0}^{d+1}$. Then by the definition of resolvable sequences and \eqref{r-pasc2},
    \begin{align*}
        f_{[\hat{0}, h)} &= f_{\langle h \rangle} - t^d = R_d - t^d = \sum_{j=0}^{d-1} \lambda_{d-1,j} R_{d-1,j}
    \end{align*}
    and
    \begin{align*}
        f_{P_J} &= f_P - t^{d+1} - r_{d+1,d}t^d = R_{d+1} - t^{d+1} - r_{d+1,d}t^d \smash{\text{\quad\raisebox{-1.0\baselineskip}{.}}}\\
                   &= \sum_{i=0}^d \lambda_{d,i} (R_{d,i}  - t^d) = \sum_{i=0}^d \lambda_{d,i} \sum_{j=i}^{d-1} \lambda_{d-1,j} R_{d-1,j} 
    \end{align*}
    Let $\EE$ be the subdivision operator associated to $P$. Then
    $$
        c_{[\hat{0}, h)} = (1+t) \cdot \EE (f_{[\hat{0}, h)}) = (1+t) \cdot \sum_{j=0}^{d-1} \lambda_{d-1,j} \EE(R_{d-1,j})
    $$
    and
    $$
        c_{P_J} = (1+t) \cdot \EE (f_{P_J}) = (1+t) \cdot \sum_{i=0}^d \lambda_{d,i} \sum_{j=i}^{d-1} \lambda_{d-1,j} \EE \left( R_{d-1,j}  \right).
    $$
    By Lemma \ref{basint}(ii), what is left to prove is that
    $$
        \sum_{j=0}^{d-1} \lambda_{d-1,j} \EE(R_{d-1,j}) \preceq \sum_{j=i}^{d-1} \lambda_{d-1,j} \EE(R_{d-1,j})
    $$
    for each $i = 0, \ldots, d-1$.
    
    To this end, define the matrix
    $$
        A = (a_{ij})_{i,j = 0}^{d-1} = \begin{cases}
            \lambda_{d-1,j} &\mbox{ if } j \geq i \\
            0 &\mbox{ otherwise}
        \end{cases}\,\quad .
    $$
    Since $A$ is $\mathrm{TP}_2$, Lemma \ref{basint}(i) implies that $A$ preserves interlacing sequences of polynomials. Since $\{ \EE(R_{d-1,i}) \}_{i=0}^{d-1}$ is an interlacing sequence of polynomials by \cite[Theorem 3.6]{branden2024totally}, the sequence
    $$
        \{ A(\EE(R_{d-1,i})) \}_{i=0}^{d-1} = \Biggl\{ \sum_{j=i}^{d-1} \lambda_{d-1,j} \EE(R_{d-1,j}) \Biggr\}_{i=0}^{d-1}
    $$
    is an interlacing sequence of polynomials whose zeros all lie in the interval $[-1,0]$.
\end{proof}

We now have all the ingredients to prove Theorem \ref{main-theorem}:

\begin{proof}[Proof of Theorem \ref{main-theorem}]
    Let $Q = P(\mathcal{H}, y, d)$. Since $Q$ is a finite poset with a greatest element $y$, Lemma \ref{main-rec} implies
    \begin{equation}
    \label{main-lemma}
        c_Q = c_{\tau(Q)} + t \sum_{h \in \mathcal{H}} c_{\langle h \rangle} .
    \end{equation}
    
    We now prove that $c_{\langle h \rangle}  \preceq c_{\tau(Q)}$ for each $h \in \mathcal{H}$. In fact, let $h \in \mathcal{H}$ be such that $d \leq \rho(h) = m < \rho(y)$ in $P$ and set $\tilde{\mathcal{H}} = \{ x \in \langle y \rangle \colon \rho(x) = m \}$, and $\tilde{Q} = P(\tilde{\mathcal{H}}, y, d)$. Since $P$ is a $\mathrm{TN}$-poset, $\tilde{Q}$ is also $\mathrm{TN}$ by \cite[Proposition 5.8]{branden2024totally}. Hence, by Lemma \ref{extra-lemma}, $c_{\langle h \rangle}  \preceq c_{\tau(\Tilde{Q})} = c_{\tau(Q)}$, which completes the proof.
\end{proof}

\begin{corollary}
\label{cor-construction-rank-selection}
    Let $P$ be a $\mathrm{TN}$-poset with a least element $\zero$. If $Q = P(\mathcal{H}, y, d)$ and $S = \{ 0 = s_0 < \cdots < s_k \leq d+1 \}$ is a subset of $\mathbb{N}$, then the chain polynomial of $Q_S$ is real-rooted and all of its zeros lie in the interval $[-1,0]$.
\end{corollary}

\begin{proof}
    Let $\langle y \rangle \subseteq P$, and let $n$ be the quasi-rank of $y$ in $P$. If $d \notin S$, then $Q_S$ is a rank-selected subposet of the $\mathrm{TN}$-poset $\langle y \rangle$, which implies that $Q_S$ is $\mathrm{TN}$ by \cite[Proposition 5.9]{branden2024totally}. Hence, the result follows.

    Otherwise, let $\Tilde{S} = S \cup \{ d+1, \ldots, n \}$. Then $P_{\Tilde{S}}$ is $\mathrm{TN}$ and $Q_S = P_{\Tilde{S}}(\mathcal{H}, y, k)$, whose chain polynomial is real-rooted and all of its zeros lie in the interval $[-1,0]$ by Theorem \ref{main-theorem}.
\end{proof}

\subsection{\texorpdfstring{$L$}--paving lattices}
\label{paving-matroids}
A family $\mathcal{H}$ of two or more sets forms a  \emph{$d$-partition} if every set in $\mathcal{H}$ has at least $d$ elements, and every $d$-element subset of $\bigcup \mathcal{H}$ is a subset of exactly one set in $\mathcal{H}$. For example, if $V$ is a finite set and $\mathcal{P}^d_V = \binom{V}{d}$, then $\mathcal{P}^d_{V}$ is a $d$-partition of $V$. Another example is
\[
    \mathcal{T} = \mathcal{T}_1 \cup \mathcal{T}_2,
\]
a $3$-partition of $[8]$ where
\[
    \mathcal{T}_1 = \{ \{1,2,3,4\}, \{1,4,5,6\}, \{2,3,5,6\}, \{1,4,7,8\}, \{2,3,7,8\} \}
\]
and $\mathcal{T}_2 = \{ T \subseteq [8] \colon |T| = 3 \}$.

A geometric lattice $L$ of rank $d+1$ is called a \emph{paving lattice}\footnote{These are the lattices of flats of paving matroids.} if $d \in \{0, 1\}$, or if $d \geq 2$ and $L$ is isomorphic to $\mathbb{B}_n (\mathcal{H}, [n], d)$ for some $d$-partition $\mathcal{H}$ of $[n]$. For example, the poset $\tau^{n-d-1}(\mathbb{B}_n)$ is a paving lattice whose set of co-atoms is $\mathcal{P}^d_{[n]}$, and the Vámos lattice \cite[Example 2.1.25]{oxley2006matroid} is a paving lattice whose set of co-atoms is $\mathcal{T}$. 

%and a poset $P(V, \mathcal{B})$ such that $(V, \mathcal{B})$ is a $t-(n, k, 1)$-design (also known as \emph{Steiner system} \cite{steiner}) is a paving lattice whose set of co-atoms is $\mathcal{B}$.

It has been conjectured that almost all geometric lattices are paving lattices~\cite[Conjecture 1.6]{mayhew2011asymptotic}, in the sense that
$$
    \lim_{n\to\infty} \left( \frac{\text{number of paving lattices with $n$ atoms}}{\text{number of geometric lattices with $n$ atoms}} \right) = 1.
$$

\begin{theorem}
\label{maincor}
    The chain polynomial of any paving lattice $L$ is real-rooted and all of its zeros lie in the interval $[-1,0]$.
    Moreover, if $S$ is a set of nonnegative integers, then $L_S$ has a $[-1,0]$-rooted chain polynomial.
\end{theorem}

\begin{proof}
    Follows directly from Theorem \ref{main-theorem} and Corollary \ref{cor-construction-rank-selection}.
\end{proof}

We now generalize the idea behind the construction of a paving lattice. Let $L$ be a finite geometric lattice with rank function $\rho$. Suppose $d \geq 1$, and  $\mathcal{H} \subseteq L$ is such that
    \begin{enumerate}[label=(\roman*)]
        \item $\hat{1} \notin \mathcal{H}$,
        \item $\rho(h)\geq d$ for all $h \in \mathcal{H}$,
        \item $\mathcal{H}$ is an antichain, and
        \item for each $x \in L$ with $\rho(x)= d$ there exists a unique $h \in \mathcal{H}$ such that $x \leq h$. 
    \end{enumerate} 
The family $\mathcal{H}$ is called a \emph{generalized $d$-partition} of $L$.

\begin{lemma}\label{d-part}
    If $L$ is a geometric lattice and $\mathcal{H}$ is a generalized $d$-partition of $L$, then $L(\mathcal{H},\hat{1},d)$ is a geometric lattice.
\end{lemma}
    
\begin{proof}
       The claim is obvious for $d=1$. Assume $d>1$. Let $\rho'$ be the rank function of $L(\mathcal{H},\hat{1},d)$, and let $\vee$ and $\vee'$ (respectively, $\wedge$ and $\wedge'$) be the joins (respectively, the meets) in $L$ and $L(\mathcal{H},\hat{1},d)$, respectively. We will prove that $L(\mathcal{H},\hat{1},d)$ is $(1)$ a lattice, $(2)$ atomistic and $(3)$ semimodular:
       \begin{enumerate}[leftmargin=*]
           \item Since $L(\mathcal{H},\hat{1},d)$ is a finite meet semi-lattice with a largest element, then, by \cite[Proposition 3.3.1]{stanley2011enumerative}, it is a lattice;
           \item Let $x \in L(\mathcal{H},\hat{1},d)$. Observe that $a$ is an atom of $L(\mathcal{H},\hat{1},d)$ if and only if it is an atom of $L$. There are three different cases to deal with:
           \begin{itemize}
               \item $\rho'(x) \leq d-1$. By definition of $L(\mathcal{H},\hat{1},d)$, if $x = \bigvee_i a_i$ for atoms $a_i \in L$, then $x = \bigvee_i' a_i$ in $L(\mathcal{H},\hat{1},d)$;
               \item $\rho'(x) = d$. Let $h_1, \ldots, h_k$ be the elements of $L$ with $\rho(h_i) = d$, $i = 1, \ldots, k$, such that $h_1, \ldots, h_k \leq x$ in $L$. Then $x = \bigvee_{i \in [k]} h_i = \bigvee_j a_j$ for some atoms $a_j$ of $L$. Therefore, $x = \bigvee_j' a_j$;
               \item $\rho'(x) = d+1$. Hence, $x = \hat{1}$, which is the join of different elements in $\mathcal{H}$ and hence a join of atoms in $L(\mathcal{H},\hat{1},d)$, by the above.
           \end{itemize}
           \item Let $x,y \in L(\mathcal{H})$. We want to prove
           \begin{equation}
           \label{smod}
                \rho'(x) + \rho'(y) \geq \rho' (x \vee ' y) + \rho' (x \wedge ' y).
           \end{equation}
           There are three different cases to deal with:
           \begin{itemize}[leftmargin=*]
                \item $x$ and $y$ are in $\mathcal{H}$. Then $\rho' (x) = \rho' (x) = d$, $\rho' (x \vee ' y) = d + 1$, and $\rho' (x \wedge ' y) \leq d-1$, which implies \eqref{smod};
                \item $x$ is not in $\mathcal{H}$ and $y$ is in $\mathcal{H}$. If $x \leq' y$, then there is nothing to prove. Otherwise, $\rho' (x \wedge ' y) \leq \rho' (x) - 1$, and $x \vee' y = \hat{1}$, from which \eqref{smod} holds;
                \item $x$ and $y$ are not in $\mathcal{H}$. We may assume $x$ and $y$ are smaller than $\hat{1}$. Then 
                  \begin{align*}
                    \rho'(x)+\rho'(y) &= \rho(x) + \rho(y) \geq  \rho (x \vee y) + \rho (x \wedge y) \\
                     &=  \rho (x \vee y) + \rho'(x \wedge' y). 
                \end{align*}
            Hence it remains to prove 
            \begin{equation}\label{fpf}
            \rho'(x \vee' y) \leq \rho (x \vee y). 
            \end{equation}
            Again, there are three different cases to deal with:
            \begin{itemize}
                \item $\rho (x \vee y) \leq d-1$. Then $x \vee y=x \vee' y$, and so \eqref{fpf} holds;
                \item $\rho (x \vee y) = d$. Then there exists $h \in \mathcal{H}$ such that $x \vee y \leq h$, and hence $\rho'(x \vee' y) \leq \rho'(h) = d$, which verifies \eqref{fpf};
                \item $\rho (x \vee y) \geq d+1$. Then $x \vee' y = \hat{1}$, which implies $\rho'(x \vee' y)=d+1$, and verifies \eqref{fpf}. 
            \end{itemize}
        \end{itemize}
    \end{enumerate}
\end{proof}
    
The geometric lattice $L(\mathcal{H},\hat{1},d)$ is referred to as an $L$\emph{-paving lattice}. For instance, when $L$ corresponds to a Boolean algebra, $L(\mathcal{H},\hat{1},d)$ is a paving lattice.

\begin{remark}
    Any geometric lattice $L$ is an $L$-paving lattice. In fact, it suffices taking $\mathcal{H}$ as the set of co-atoms of $L$ and $d = \rho(L)-1$.
\end{remark}

\begin{theorem}
\label{rank-uniform-matroids}
    Let $L$ be a rank uniform geometric lattice. Then the chain polynomial of any $L$-paving lattice $Q$ is real-rooted, and all of its zeros lie in the interval $[-1,0]$. Moreover if $S$ is a set of nonnegative integers, then $Q_S$ has a $[-1,0]$-rooted chain polynomial.
\end{theorem}

\begin{proof}
    The first part of the proof follows from Theorem \ref{main-theorem} and from the fact that $L$ is a $\mathrm{TN}$-poset (Theorem \ref{rankunigeo}). The second part follows from Corollary \ref{cor-construction-rank-selection}.
\end{proof}

\begin{example}
    By Theorem~\ref{rank-uniform-matroids}, the chain polynomials of $\mathbb{B}_n(q)$-paving lattices and $\mathbb{A}_n(q)$-paving lattices are real-rooted.
\end{example}

\section{Single-element extensions}
\label{SEE}

In this section, we define single-element extension of geometric lattices and prove that any principal single-element extension of $\mathbb{B}(E)$ and $\tau(\mathbb{B}(E))$ has real-rooted chain polynomial. For easy of exposition, if $X$ is a set and $e \notin X$, we denote $X \cup \{e\}$ by $X \cup e$.

Let $L$ be a geometric lattice with rank function $\rho$. A \emph{modular cut} $\mathcal{M}$ of $L$ is a collection of elements of $L$ that satisfies the following:
\begin{enumerate}[label=(\roman*)]
    \item if $x \in \mathcal{M}$, and $y \in L$ is such that $x \leq y$, then $y \in \mathcal{M}$;
    \item if $x, y \in \mathcal{M}$ is a \emph{modular pair} of $L$, that is, $\rho(x) + \rho(y) = \rho(x \wedge y) + \rho(x \vee y)$, then $x \wedge y\in \mathcal{M}$.
\end{enumerate}

For $x_1, \ldots, x_k \in L$, let
$
    \{x_1, \ldots, x_k\}^+ \coloneqq \{ x \in L \colon x \geq x_i \mbox{ for some } i = 1, \ldots, k \}.
$
By definition, a modular cut $\mathcal{M} \neq \emptyset$ is determined by its minimal elements, hence $\mathcal{M} = \{x_1, \ldots, x_k\}^+$ for some $x_1, \ldots, x_k \in L$. When $k=1$, $\mathcal{M}$ is called a \emph{principal modular cut} of $L$, and we write $x_1^+$ instead of $\{x_1\}^+$.

Let $A$ be the set of atoms of $L$. We may identify each $\hat{0} \neq x \in L$ with the set $X = \{ a \in A \colon a\leq x \}$. Given a modular cut $\mathcal{M}$ of $L$ and $e \notin A$, the \emph{single-element extension} $L +_\mathcal{M} e$ is defined as the geometric lattice whose elements fall into the three following disjoint classes \cite[Corollary 7.2]{oxley2006matroid}:

\begin{enumerate}[label=(\roman*)]
    \item elements $X$ of $L$ that are not in $\mathcal{M}$;
    \item elements $X \cup e$, where $X \in \mathcal{M}$;
    \item elements $X \cup e$, where $X \in L$ is not in $\mathcal{M}$ and there is no $Y \in \mathcal{M}$ such that $\rho(Y) = \rho(X) + 1$ and $X \subseteq Y$.
\end{enumerate}
When $\mathcal{M} = X^+$ is a principal modular cut of $L$, we write $L+_X e$. For example, consider the Boolean algebra $\mathbb{B}_3$. Then, $\mathcal{M}_1 = \{1\}^+$ and $\mathcal{M}_2 = \{1,2\}^+$ are modular cuts. The Hasse diagrams of of $\mathbb{B}_3$, $\mathbb{B}_3 +_{\{1\}} 4$ and $\mathbb{B}_3 +_{\{1,2\}} 4$ are given in Figure \ref{fig:ex-sing-el}.

\begin{figure}[H]
    \centering
    \begin{tikzpicture}[scale=0.52,every node/.style={scale=0.65}, line cap=round,line join=round,>=triangle 45,x=1cm,y=1cm]]
        \clip(-4,-4) rectangle (20,4);
        \draw [line width=0.5pt] (-2,-1)-- (0,-3);
        \draw [line width=0.5pt] (0,-1)-- (0,-3);
        \draw [line width=0.5pt] (2,-1)-- (0,-3);
        \draw [line width=0.5pt] (0,1)-- (-2,-1);
        \draw [line width=0.5pt] (0,1)-- (2,-1);
        \draw [line width=0.5pt] (-2,1)-- (-2,-1);
        \draw [line width=0.5pt] (-2,1)-- (0,-1);
        \draw [line width=0.5pt] (2,1)-- (2,-1);
        \draw [line width=0.5pt] (2,1)-- (0,-1);
        \draw [line width=0.5pt] (0,3)-- (0,1);
        \draw [line width=0.5pt] (0,3)-- (-2,1);
        \draw [line width=0.5pt] (0,3)-- (2,1);
        \draw [line width=0.5pt] (5,-1)-- (7,1);
        \draw [line width=0.5pt] (7,1)-- (9,-1);
        \draw [line width=0.5pt] (9,-1)-- (7,-3);
        \draw [line width=0.5pt] (7,-3)-- (7,-1);
        \draw [line width=0.5pt] (7,-1)-- (5,1);
        \draw [line width=0.5pt] (5,1)-- (5,-1);
        \draw [line width=0.5pt] (5,-1)-- (7,-3);
        \draw [line width=0.5pt] (7,-1)-- (9,1);
        \draw [line width=0.5pt] (9,1)-- (9,-1);
        \draw [line width=0.5pt] (9,1)-- (7,3);
        \draw [line width=0.5pt] (7,3)-- (7,1);
        \draw [line width=0.5pt] (5,1)-- (7,3);
        \draw [line width=0.5pt] (12,1)-- (15,3);
        \draw [line width=0.5pt] (15,3)-- (14,1);
        \draw [line width=0.5pt] (16,1)-- (15,3);
        \draw [line width=0.5pt] (15,3)-- (18,1);
        \draw [line width=0.5pt] (12,-1)-- (15,-3);
        \draw [line width=0.5pt] (15,-3)-- (14,-1);
        \draw [line width=0.5pt] (16,-1)-- (15,-3);
        \draw [line width=0.5pt] (15,-3)-- (18,-1);
        \draw [line width=0.5pt] (12,-1)-- (12,1);
        \draw [line width=0.5pt] (12,1)-- (14,-1);
        \draw [line width=0.5pt] (12,-1)-- (14,1);
        \draw [line width=0.5pt] (14,1)-- (16,-1);
        \draw [line width=0.5pt] (14,-1)-- (16,1);
        \draw [line width=0.5pt] (16,1)-- (16,-1);
        \draw [line width=0.5pt] (16,-1)-- (18,1);
        \draw [line width=0.5pt] (18,1)-- (18,-1);
        \draw [line width=0.5pt] (18,-1)-- (12,1);
        \draw (-0.2,-3.3) node[anchor=north west] {$\emptyset$};
        \draw (6.909130690063523,-3.1175894455267446) node[anchor=north west] {$\emptyset$};
        \draw (14.898949859786027,-3.109134610426509) node[anchor=north west] {$\emptyset$};
        \draw (-2.9,-1.1053386916706984) node[anchor=north west] {$\{ 1 \}$};
        \draw (-0.1,-1.1053386916706984) node[anchor=north west] {$\{ 2 \}$};
        \draw (1.8277747948220149,-1.1053386916706984) node[anchor=north west] {$\{ 3 \}$};
        \draw (-3.6,1.0929184343905283) node[anchor=north west] {$\{ 1,2 \}$};
        \draw (-0.8,0.5) node[anchor=north west] {$\{ 1,3 \}$};
        \draw (2.0306908372276657,1.1098281045909992) node[anchor=north west] {$\{ 2,3 \}$};
        \draw (-1.0, 4.0) node[anchor=north west] {$\{ 1,2,3 \}$};
        \draw (3.5,-0.9108774843652822) node[anchor=north west] {$\{ 1,4 \}$};
        \draw (9.031294300222621,-0.8939678141648112) node[anchor=north west] {$\{ 3 \}$};
        \draw (7,-0.8939678141648112) node[anchor=north west] {$\{ 2 \}$};
        \draw (6.1,0.35) node[anchor=north west] {$\{ 1,3,4 \}$};
        \draw (3.2,2) node[anchor=north west] {$\{ 1,2,4 \}$};
        \draw (9.031294300222621,2) node[anchor=north west] {$\{ 2,3 \}$};
        \draw (5.7, 4.0) node[anchor=north west] {$\{ 1,2,3,4 \}$};
        \draw (10.9,-0.9108774843652822) node[anchor=north west] {$\{ 1 \}$};
        \draw (12.9,-0.9108774843652822) node[anchor=north west] {$\{ 2 \}$};
        \draw (16.02344292811734,-0.8939678141648112) node[anchor=north west] {$\{ 3 \}$};
        \draw (18.027238846873143,-0.8939678141648112) node[anchor=north west] {$\{ 4 \}$};
        \draw (10.2,1.4902956841015962) node[anchor=north west] {$\{ 1,2,4 \}$};
        \draw (12.5,1.4987505192018318) node[anchor=north west] {$\{ 1,3 \}$};
        \draw (16.048807433418048,1.4649311788008899) node[anchor=north west] {$\{ 2,3 \}$};
        \draw (18.044148517073616,1.4902956841015962) node[anchor=north west] {$\{ 3,4 \}$};
        \draw (13.7,4) node[anchor=north west] {$\{ 1,2,3,4 \}$};
        \begin{scriptsize}
            \draw [fill=black] (0,-3) circle (2.5pt);
            \draw [fill=black] (0,-1) circle (2.5pt);
            \draw [fill=black] (2,-1) circle (2.5pt);
            \draw [fill=black] (-2,-1) circle (2.5pt);
            \draw [fill=black] (0,1) circle (2.5pt);
            \draw [fill=black] (2,1) circle (2.5pt);
            \draw [fill=black] (-2,1) circle (2.5pt);
            \draw [fill=black] (0,3) circle (2.5pt);
            \draw [fill=black] (5,-1) circle (2.5pt);
            \draw [fill=black] (7,-1) circle (2.5pt);
            \draw [fill=black] (9,-1) circle (2.5pt);
            \draw [fill=black] (7,-3) circle (2.5pt);
            \draw [fill=black] (5,1) circle (2.5pt);
            \draw [fill=black] (7,1) circle (2.5pt);
            \draw [fill=black] (9,1) circle (2.5pt);
            \draw [fill=black] (7,3) circle (2.5pt);
            \draw [fill=black] (12,-1) circle (2.5pt);
            \draw [fill=black] (14,-1) circle (2.5pt);
            \draw [fill=black] (16,-1) circle (2.5pt);
            \draw [fill=black] (18,-1) circle (2.5pt);
            \draw [fill=black] (15,-3) circle (2.5pt);
            \draw [fill=black] (12,1) circle (2.5pt);
            \draw [fill=black] (14,1) circle (2.5pt);
            \draw [fill=black] (16,1) circle (2.5pt);
            \draw [fill=black] (18,1) circle (2.5pt);
            \draw [fill=black] (15,3) circle (2.5pt);
        \end{scriptsize}
    \end{tikzpicture}
    \caption{The Hasse diagrams of $\mathbb{B}_3$, $\mathbb{B}_3 +_{\{1\}} 4$ and $\mathbb{B}_3 +_{\{1,2\}} 4$, respectively.}
    \label{fig:ex-sing-el}
\end{figure}
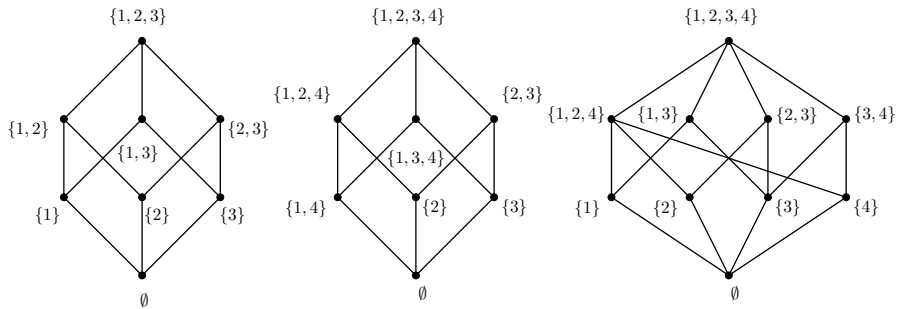

Any geometric lattice may be constructed from a Boolean lattice by repeatedly single-element extensions. This technique was used by Blackburn, Crapo, and Higgs \cite{catalog} to find all geometric lattices with at most $8$ atoms.  It is not hard to prove that Conjecture~\ref{conj1} is true for geometric lattices of rank at most three. We provide a proof here for reference. 

\begin{theorem}
\label{thr5.1}
    Conjecture~\ref{conj1} is true for all geometric lattices of ranks $1$, $2$ or $3$.
\end{theorem}

\begin{proof}
The chain polynomials of the geometric lattices of ranks $1$ and $2$ are given by $c_1 = (1+t)^2$ and $c_2 = (1+mt)(1+t)^2$ (where $m$ is the number of co-atoms of the geometric lattice), respectively, which are clearly real-rooted. 

The chain polynomial of any geometric lattice $L$ of rank $3$ is given by $c_L = [1 + (m_1 + m_2)t + et^2](1+t)^2$, where $m_i$ is the number of rank-$i$ elements and $e$ is the number of edges between rank-$1$ elements and rank-$2$ elements in the Hasse diagram of $L$. Observe that $e \leq m_1 m_2$. Since the discriminant of $p = 1 + (m_1 + m_2)t + et^2$ is given by $(m_1 - m_2)^2 + 4 (m_1 m_2 - e) \geq 0$, $c_L$ is real-rooted.
\end{proof}

For example, the Fano \cite[Example 1.5.7]{oxley2006matroid} and the Pappus \cite[Example 1.5.15]{oxley2006matroid} lattices are rank-$3$ geometric lattices. Hence, by Theorem \ref{thr5.1}, their chain polynomials are real-rooted.

Now, consider the Boolean algebras $\mathbb{B}(E)$, $|E| \geq 4$. Our goal is to study the chain polynomial of a single-element extension of $\mathbb{B}(E)$, so we must understand how to describe its modular cuts. The following proposition is straightforward:

\begin{proposition}
\label{modular-cuts-Boolean}
    Each non-empty modular cut $\mathcal{M}$ of $\mathbb{B}(E)$ is principal.
\end{proposition}

\begin{proof}
    Follows from the modularity of $\mathbb{B}(E)$.
\end{proof}

Let $\mathcal{M}$ be a modular cut of $\mathbb{B}(E)$, and let $e \notin E$. Some cases are easy to deal with:
\begin{itemize}[leftmargin=*]
    \item if $\mathcal{M} = \emptyset$, then $\mathbb{B}(E) +_\mathcal{M} e = \mathbb{B}(E \cup e)$;
    \item if $\mathcal{M} = \mathbb{B}(E)$ or $\mathcal{M} = \{ a \}^+$, $a \in E$, then $\mathbb{B}(E) +_\mathcal{M} e \cong \mathbb{B}(E)$;
    \item if $\mathcal{M} = \{E\}$, then $\mathbb{B}(E) +_\mathcal{M} e \cong \tau(\mathbb{B}(E \cup e))$.
\end{itemize}
By \cite[Proposition 5.8]{athanasiadis2022chain} (and independently proven by us in Theorem~\ref{maincor}), the chain polynomials of all the lattices listed above are real-rooted. This suggests that $\mathbb{B}(E) +_\mathcal{M} e$ has a real-rooted chain polynomial for each modular cut $\mathcal{M}$ of $\mathbb{B}(E)$. In fact, we now provide the theory necessary to prove this result.

By Proposition \ref{modular-cuts-Boolean}, if $\mathcal{M}$ is not of the form listed above, each single-element extension is given by to $\mathbb{B}(E) +_\mathcal{M} e$, where $\mathcal{M} = X^+$ for some $X \subseteq E$ such that $2 \leq |X| \leq n-1$. Hence, each element $Y$ of $\mathbb{B}(E) +_\mathcal{M} e$ falls into the following disjoint classes:
\begin{enumerate}[label=(\roman*)]
    \item [(i)] $Y \subseteq E \cup e$ such that $X \cup e$ is not a subset of $Y$. In this case, the rank of $Y$ is $|Y|$;
    \item [(ii)] $Y \subseteq E \cup e$ such that $X \cup e \subseteq Y$. In this case, the rank of $Y$ is $|Y| - 1$.
\end{enumerate}

It follows that, if $\emptyset \neq X \subseteq E$, then $\mathbb{B}(E) +_X e$ may be expressed as the \emph{direct product}  $\mathbb{B}(E) +_X e= \tau (\mathbb{B}(X \cup e)) \times \mathbb{B}(E \setminus X)$, see \cite[Chapter 3]{stanley2011enumerative}. We will use this fact and the next lemmas to prove that the chain polynomial of $\mathbb{B}(E) +_X e$ is real-rooted. Wagner \cite{MR1145841} defined the \emph{diamond product} as follows.  Let $\mathcal{E}: \mathbb{R}[t] \to \mathbb{R}[t]$ be the linear operator defined on the binomial basis by
    $$
        \mathcal{E} \left( \binom{t}{k} \right) = t^k
    $$
    for all $k \in \mathbb{N}$. Then the \emph{diamond product} of two polynomials $f$ and $g$ is defined by 
    $$
    (f \diamond g)=  \mathcal{E} \left( \mathcal{E}^{-1}(f) \mathcal{E}^{-1}(g)\right). 
    $$ 

\begin{lemma}\cite{MR1145841}
\label{lemma1}
    If $f$ and $g$ are real-rooted polynomials whose zeros all lie in the interval $[-1,0]$, then so is the polynomial $(f \diamond g)$.
\end{lemma}
\begin{lemma}
\label{lemma2}
    Let $(P, \leq_P)$ and $(Q, \leq_Q)$ be two finite posets with a least and a greatest elements, respectively, such that $|P|, |Q| \geq 2$. Let
    \[
        p_P = \sum_{j \geq 0} |\{\zero_P =x_0 < x_1<\cdots < x_{j+1}=\hat{1}_P\}|t^{j+1}.
    \]
    Then $p_{P \times Q} = (p_P \diamond p_Q)$. 
    
\end{lemma}
\begin{proof}
    Define
    $$
        Z_P (n) \coloneqq |\{ \hat{0}_P = x_0 \leq x_1 \leq \cdots \leq x_n = \hat{1}_P \}|,
    $$
    the \emph{zeta polynomial} of $P$ \cite[Section 3.12]{stanley2011enumerative}. The number of multichains $\zero_P = x_0 \leq x_1 \leq \cdots \leq x_n= \one_P$ with $k+1$ distinct elements $\zero_P = \hat{x}_0 < \hat{x}_1 < \cdots < \hat{x}_k < \hat{x}_{k+1} = \one_P$ is equal to 
    $
        w_k(P) \binom n {k+1}, 
    $
    where 
    $$
        w_k(P)= |\{\zero_P < \hat{x}_1 < \cdots < \hat{x}_k<\one_P\}|.
    $$
    Hence,
    $$
        Z_P(n)= \sum_{k \geq 0} w_k(P) \binom n {k+1}. 
    $$
    Recall that $\EE\left( \binom t k \right)= t^k$ for all $k \in \NN$. Hence,
    \[
        \EE(Z_P) =\sum_{k \geq 0} w_k(P) t^{k+1} = p_P.
    \]
    Observe that
    \begin{align*}
        Z_{P \times Q}(n) &= |\{ (\hat{0}_P, \hat{0}_Q) = (x_0, y_0) \leq (x_1, y_1) \leq \cdots \leq (x_n, y_n) = (\hat{1}_P, \hat{1}_Q) \}|\\
        &= Z_P(n) Z_Q(n).
    \end{align*}
    Thus,
    $$
        (p_{P} \diamond p_{Q}) = \mathcal{E} (\mathcal{E}^{-1}(p_{P}) \mathcal{E}^{-1}(p_{Q})) = \mathcal{E}(Z_P Z_Q) = \mathcal{E}(Z_{P \times Q}) = p_{P \times Q}.
    $$
\end{proof}

\begin{lemma}
\label{lemma5}
    Let $P$ be a finite poset with a least and a greatest elements. Then
    \[
        c_P = \frac{(1+t)^2}{t}p_P.
    \]
    In particular, $p_P$ is $[-1,0]$-rooted if and only if $c_P$ is $[-1,0]$-rooted.
\end{lemma}
\begin{proof}
    Obvious from the definition of $c_P$ and $p_P$.
\end{proof}

Recall that the chain of polynomial of a poset $P$ is the $f$-polynomial of the order complex $\Delta_P$, consisting of all chains in  $P$. The $h$-polynomial of a simplicial complex $\Delta$ is related to the $f$-polynomial by 
$$
f_\Delta= (1+t)^d h_\Delta\left(\frac t {1+t}\right).
$$
Hence it follows that if the coefficients of the $h_\Delta$ are nonnegative, then any real zero of $f_\Delta$ is necessarily located in the interval $[-1,0]$. Non-negativity of the coefficients of $h_\Delta$ holds for many important classes of complexes, for example 
Cohen-Macaulay complexes which include order complexes of semimodular lattices as well as distributive lattices \cite{BGS-81}, but also order complexes of $\mathrm{TN}$-posets \cite{branden2024totally}. By Lemmas \ref{lemma2} and \ref{lemma4} we conclude
\begin{proposition}
Suppose $P$ and $Q$ are posets whose chain polynomials are real-rooted with all zeros located in the interval $[-1,0]$, then so is the chain polynomial of $P \times Q$.  
\end{proposition}

\begin{lemma}
    Let $P$ be a finite poset with a least and a greatest elements of quasi-rank $d+1$. Then $p_P$ satisfies the following relation:
    \[
        p_P  = p_{\tau (P)} + t \sum_{\substack{h \in P \\ \rho(h) = d}} p_{\langle h \rangle}.
    \]
\end{lemma}

\begin{proof}
    Follows directly from Lemmas \ref{main-rec} and \ref{lemma5}.
\end{proof}

\begin{lemma}\cite[Theorem 3]{MR2104673}
\label{lemma4}
    Let $h$ be a $[-1,0]$-rooted polynomial and $f$ be a real-rooted polynomial. Then $(f \diamond h)$ is real-rooted, and, if $g \preceq f$, then $(g \diamond h) \preceq (f \diamond h)$.
\end{lemma}

\begin{comment}
\begin{lemma}[\cite{branden2006linear}]
\label{lemma3}
    If
    $$
        f = \sum_{k=0}^d h_k t^k (1+t)^{d-k}
    $$
    has $h_k \geq 0$ for all $0 \leq k \leq d$, then all zeros of $\mathcal{E}(f)$ are real, simple and located in $[-1,0]$. In particular, the $h$-polynomial of a Cohen-Macaulay poset is real-rooted.
\end{lemma}

Now, we can prove the following:

\begin{theorem}
    If the $h$-polynomials of the order complexes of the posets $P$ and $Q$ have nonnegative coefficients, then the chain polynomial of $P \times Q$ is real-rooted.
\end{theorem}
\begin{proof}
    By Lemma \ref{lemma5}, it suffices to prove that $p_{P \times Q}$ is real-rooted. By Lemma \ref{lemma2},
    \[
        p_{P \times Q} = (p_P \diamond p_Q).
    \]
    
    By hypothesis, the $h$-polynomials of the order complexes of the posets $P$ and $Q$ have nonnegative coefficients, which implies that $c_P$ and $c_Q$ are $[-1,0]$-rooted by Lemma \ref{lemma3}. Hence, $p_P$ and $p_Q$ are $[-1,0]$-rooted, which implies that $(p_P \diamond p_Q) $ is real-rooted by Lemma \ref{lemma4}.
\end{proof}
\end{comment}

\begin{corollary}
\label{SEE-thr}
    Let $E$ be a finite set and $e \notin E$. The chain polynomial of $\mathbb{B}(E) +_\mathcal{M} e$ is $[-1,0]$-rooted for any modular cut $\mathcal{M}$ of $\mathbb{B}(E)$.
\end{corollary}

\begin{proof}
    The cases where $\mathcal{M} = \emptyset$ or $\mathcal{M} = X^+$, $|X| \in \{ 0, 1, |E|\}$, were already discussed. Suppose now that $\mathcal{M} = X^+$, $2 \leq |X| < |E|$. Then $\mathbb{B}(E) +_X e = \tau(\mathbb{B}(X \cup e)) \times \mathbb{B}(E \setminus X)$. So, by Lemma~\ref{lemma2},
    \[
        p_{\mathbb{B}(E) +_X e} = \left( p_{\tau(\mathbb{B}(X \cup e))} \diamond p_{\mathbb{B}(E \setminus X)} \right) .
    \]
    By Theorem \ref{maincor}, $c_{\tau(\mathbb{B}(X \cup e))} $ and $c_{\mathbb{B}(E \setminus X)} $ are $[-1,0]$-rooted. So, the polynomials $p_{\tau(\mathbb{B}(X \cup e))} $ and $p_{\mathbb{B}(E \setminus X)} $ are $[-1,0]$-rooted by Lemma \ref{lemma5}, which implies that $p_{\mathbb{B}(E) +_X e}$ is real-rooted by Lemmas~\ref{lemma2} and \ref{lemma1}.
\end{proof}

By \cite[Proposition 3.3]{ernest2023extensions}, not all modular cuts of $\tau^k(\mathbb{B}(E))$, $1
\leq k < |E|-1$, are principal. However, its principal modular cuts are easy to describe:

\begin{lemma}
\label{new-lemma}
    Let $\mathcal{M} = X^+$ be a principal modular cut of $\tau^k(\mathbb{B}(E))$, $1 \leq k < |E|$. Then
    \[
        \tau^k(\mathbb{B}(E)) +_\mathcal{M} e \cong \tau^k (\mathbb{B}(E) +_{\mathcal{N}} e),
    \]
    where $\mathcal{N} = X^+$ is a principal modular cut of $\mathbb{B}(E)$. In particular, for $k=1$, the set $\mathcal{H}$ of co-atoms of $\tau(\mathbb{B}(E)) +_\mathcal{M} e$ is given by
    \[
        \mathcal{H} = \{ A \times (E \setminus X) \colon A \in \mathcal{A} \} \sqcup \{ (X \cup e) \times B \colon B \in \mathcal{B} \},
    \]
    where $\mathcal{A}$ and $\mathcal{B}$ are the sets of co-atoms of $\tau(\mathbb{B}(X \cup e))$ and $\mathbb{B}(X \setminus E)$, respectively.
\end{lemma}

\begin{proof}
    Follows from \cite[Proposition 7.3.3]{theory-of-matroids}.
\end{proof}

\begin{corollary}
\label{SEE-cor}
    Let $E$ be a finite set and $e \notin E$. If $\mathcal{M} = \emptyset$ or $\mathcal{M}$ is a principal modular cut of $\tau(\mathbb{B}(E))$, then the chain polynomial of $\tau(\mathbb{B}(E)) +_\mathcal{M} e$ is $[-1,0]$-rooted.
\end{corollary}
    
\begin{proof}
    If $\mathcal{M} = \emptyset$, then $\tau(\mathbb{B}(E)) +_\mathcal{M} e$ is a paving lattice whose set of co-atoms is $\mathcal{H} = \{ E \} \cup \left\{ A \cup e \colon A \in \mathcal{P}^{|E|-2}_{E} \right\}$, a $(|E|-1)$-partition of $E \cup e$. Hence, its chain polynomial is real-rooted by Theorem~\ref{maincor}.

    If $\mathcal{M} = \tau(\mathbb{B}(E))$, then $\tau(\mathbb{B}(E)) +_\mathcal{M} e \cong \tau(\mathbb{B}(E))$, whose chain polynomial is real-rooted by \cite[Proposition 5.8]{athanasiadis2022chain}.
    
    If $\mathcal{M} = X^+$, $|X| \geq 1$, is a modular cut of $\tau(\mathbb{B}(E))$, let $\mathcal{N} = X^+$ be a modular cut of $\mathbb{B}(E)$. By Lemma \ref{new-lemma},
    \begin{align*}
        p_{\tau(\mathbb{B}(E)) +_\mathcal{M} e} &= p_{\mathbb{B}(E) +_\mathcal{N} e} - t \sum\limits_{H \in \mathcal{H}} p_{\langle H \rangle} \\
        &= p_{\tau(\mathbb{B}(X \cup e))} \diamond p_{\mathbb{B}(E \setminus X)} - t \left( \sum\limits_{A \in \mathcal{A}} p_{\langle A \times (E \setminus X) \rangle} + \sum\limits_{B \in \mathcal{B}} p_{\langle (X \cup e) \times B \rangle} \right)\smash{\text{\raisebox{1.5\baselineskip}{.}}}
    \end{align*}
    By Lemma \ref{basint}(iv), to prove the real-rootedness of $p_{\tau(\mathbb{B}(E)) +_\mathcal{M} e}$, it suffices to show that
    \[
         \sum\limits_{A \in \mathcal{A}} p_{\langle A \times (E \setminus X) \rangle} + \sum\limits_{B \in \mathcal{B}} p_{\langle (X \cup e) \times B \rangle} \preceq \left( p_{\tau(\mathbb{B}(X \cup e))} \diamond p_{\mathbb{B}(E \setminus X)} \right).
    \]
    
    In fact, by Lemma \ref{lemma2}, for each $A \in \mathcal{A}$ and $B \in \mathcal{B}$,
    \[
        p_{\langle A \times (E \setminus X) \rangle} = \left( p_{\langle A \rangle} \diamond p_{\mathbb{B}(E \setminus X)} \right) \quad \mbox{and} \quad p_{\langle (X \cup e) \times B \rangle} = \left( p_{\mathbb{B}(X \cup e)} \diamond p_{\langle B \rangle} \right).
    \]
    By Theorem~\ref{maincor},
    \[
        p_{\langle A \rangle} \preceq  p_{\tau(\mathbb{B}(X \cup e))} \quad \mbox{and} \quad p_{\langle B \rangle} \preceq p_{\mathbb{B}(E \setminus X)}
    \]
    for every $A \in \mathcal{A}$ and $B \in \mathcal{B}$. Hence, Lemma \ref{lemma4} implies
    \[
        \left( p_{\langle A \rangle} \diamond p_{\mathbb{B}(E \setminus X)} \right), \left( p_{\tau(\mathbb{B}(X \cup e))} \diamond p_{\langle B \rangle} \right) \preceq \left( p_{\tau(\mathbb{B}(X \cup e))} \diamond p_{\mathbb{B}(E \setminus X)} \right)
    \]
    for every $A \in \mathcal{A}$ and $B \in \mathcal{B}$, which implies
    \[
         \sum\limits_{A \in \mathcal{A}} p_{\langle A \times (E \setminus X) \rangle} + \sum\limits_{B \in \mathcal{B}} p_{\langle (X \cup e) \times B \rangle} \preceq \left( p_{\tau(\mathbb{B}(X \cup e))} \diamond p_{\mathbb{B}(E \setminus X)} \right)
    \]
    by Lemma \ref{basint}(iii). The result now follows from Lemma \ref{basint}(iv).
\end{proof}

\section{A conjecture on interlacing for chain polynomials}
\label{sec-counterexample}

Recall that, given an $(n-1)$-dimensional (finite, abstract) simplicial complex $\Delta$, its $h$\emph{-polynomial} is defined as
\[
    h_\Delta = (1-t)^n f_\Delta \left( \frac{t}{1-t} \right).
\]

Given a finite $(n-1)$-dimensional poset $P$, its \emph{order complex} is defined as the simplicial complex $\Delta(P)$ of all chains in $P$. Clearly, $f_{\Delta(P)} = c_P$. Hence,
\[
    h_{\Delta (P)} = (1-t)^n c_P \left( \frac{t}{1-t} \right).
\]

In~\cite[Conjecture 5.1]{athanasiadis2022chain}, Athanasiadis and Kalampogia-Evangelinou conjectured that the polynomial $h_L$ is real-rooted and it is interlaced by the Eulerian polynomial $A_n$ for every geometric lattice $L$ of rank $n$. We finish this paper by showing that this conjecture is false.

Recall that the \emph{descent number} $\des(\sigma)$ of a permutation $\sigma = \sigma_1 \sigma_2 \cdots \sigma_n$, written in one line notation, in the symmetric group $\sym_n$ is the number of indices $i \in [n-1]$ such that $\sigma_i > \sigma_{i+1}$. The \emph{inversion number} $\inv(\sigma)$ of $\sigma$ is the number of pairs $(i,j) \in [n] \times [n]$ such that $i < j$ and $\sigma_i > \sigma_j$. The $n$th Eulerian polynomial is defined by
$$
A_n= \sum_{\pi \in \sym_n}t^{\des(\sigma)}. 
$$

\begin{proposition}
If $n \geq 3$, then $h_{\mathbb{B}_n(q)}$ is not interlaced by $A_n$ for all $q$ sufficiently large. 
\end{proposition}
\begin{proof}
  By~\cite[Proposition 3.1]{athanasiadis2022chain},
    $$
        h_{\mathbb{B}_n(q)} = \sum_{\substack{\sigma \in S_n }} q^{\inv(\sigma)}t^{\des(\sigma)} = A_n(t;q).
    $$
    Notice that $\inv(\sigma)$ is maximized for the permutation $\hat{\sigma} = n (n-1) \cdots  2 1$, and $\inv(\hat{\sigma}) = n(n-1)/2$. Hence
    \[
        \lim_{q \rightarrow \infty} \left( q^{-\frac{n(n-1)}{2}} A_n(t;q) \right) = t^{n-1}
    \]
    for every $t \in \mathbb{R}$. Since $A_n$ does not interlace $t^{n-1}$ when $n \geq 3$, $h_{\mathbb{B}_n(q)}$ is not interlaced by $A_n$ for $q$ sufficiently large.
\end{proof}

\bibliographystyle{plain}
\bibliography{references}
\end{document}